 \documentclass[final,5p,twocolumn, authoryear]{elsarticle}

\usepackage{amssymb}
\usepackage{amsmath, amssymb}
\usepackage{amsfonts}
\usepackage{amsthm}
\usepackage{graphicx}
\usepackage{enumerate}
\usepackage[font = small]{caption}
\usepackage{ifthen}
\usepackage{multicol}
\usepackage{float}
\usepackage{fancyhdr}
\usepackage[usenames, dvipsnames]{color}
\usepackage[lofdepth,lotdepth]{subfig}
\graphicspath{{figures/}}
\usepackage{mathtools}
\usepackage[normalem]{ulem}

 \usepackage{tikz}
 \usetikzlibrary{decorations.pathmorphing,shadings}
\usetikzlibrary{arrows}
 \usetikzlibrary{plotmarks}
 \usepackage{pgfplots}

 \usepackage{cancel} 
 
\newtheorem{theorem}{Theorem}[section]
\newtheorem{lemma}[theorem]{Lemma}
\newtheorem{applemma}{Lemma}
\newtheorem{proposition}[theorem]{Proposition}
\newtheorem{corollary}[theorem]{Corollary}
\newtheorem{rem}{Remark}

\newtheorem{ass}{Assumption}
\newtheorem{definition}{Definition}[section]

\newboolean{showcomments}
\setboolean{showcomments}{true}

\newcommand{\newedit}[1]{\ifthenelse{\boolean{showcomments}}
{\textcolor{ForestGreen}{#1}}{}}
\newcommand{\bassam}[1]{\ifthenelse{\boolean{showcomments}}
{\textcolor{Blue}{(Bassam says: #1)}}{}}
\newcommand{\emma}[1]{\ifthenelse{\boolean{showcomments}}
{\textcolor{VioletRed}{(Emma says: #1)}}{}}
\newcommand{\ifneeded}[1]{\ifthenelse{\boolean{showcomments}}
{\textcolor{Gray}{#1}}{}}

\newboolean{showedit}
\setboolean{showedit}{true}
\newcommand{\edit}[1]{\ifthenelse{\boolean{showedit}}
{\textcolor{Blue}{#1}}{}}

\newcommand{\rel}{\mathrm{Re}\{\lambda_2\}}

\newcommand{\rell}{\mathrm{Re}\{\lambda_l\}}
\newcommand{\imll}{\mathrm{Im}\{\lambda_l\}}

\newcommand{\Gg}{\mathcal{G}} 
\newcommand{\Ggn}{\mathcal{G}_N} 
\newcommand{\Vv}{\mathcal{V}} 
\newcommand{\Ee}{\mathcal{E}} 
\newcommand{\Aa}{\mathcal{A}} 
\newcommand{\nth}{$n^{\text{th}}$ }
\newcommand{\ddt}{\frac{\mathrm{d}}{\mathrm{d}t}}

\usepackage{array}
\newcolumntype{L}[1]{>{\raggedright\let\newline\\\arraybackslash\hspace{0pt}}m{#1}}
\newcolumntype{C}[1]{>{\centering\let\newline\\\arraybackslash\hspace{0pt}}m{#1}}
\newcolumntype{R}[1]{>{\raggedleft\let\newline\\\arraybackslash\hspace{0pt}}m{#1}}
\usepackage{subfig}

\definecolor{gray3}{rgb}{0.75, 0.75, 0.75}
\definecolor{gray2}{rgb}{0.5, 0.5, 0.5}
\definecolor{gray1}{rgb}{0.25, 0.25, 0.25}
\definecolor{gray0}{rgb}{0.15, 0.15, 0.15}

\definecolor{emmagreen1}{rgb}{0, 0.5, 0.1}
\definecolor{emmaorange1}{rgb}{0.99, 0.5, 0}
\definecolor{emmablue1}{rgb}{0, 0.25, 0.5}
\definecolor{emmaskyblue1}{rgb}{0.4, 0.8, 0.95}
\definecolor{emmapurple1}{rgb}{0.5, 0.25, 0.6}

\pgfplotscreateplotcyclelist{mycolorlist3}{%
emmagreen1\\%
emmaorange1\\%
emmablue1\\%
gray3, densely dashed\\%
gray3, densely dashed\\%
gray3, densely dashed\\%
}

\pgfplotscreateplotcyclelist{mygraylist}{%
black\\%
gray1, dashed\\%
gray3\\%
gray2\\%
gray2,dashed\\%
gray3,dashed\\%
}

\usepackage[hyphens]{url}

 \biboptions{sort&compress}

\journal{Automatica}

\begin{document}

\begin{frontmatter}



\title{Scale Fragilities in Localized Consensus Dynamics }


\author[Auth1]{Emma Tegling}\ead{emma.tegling@control.lth.se}   
\author[Auth2]{Bassam Bamieh}\ead{bamieh@engineering.ucsb.edu}               
\author[Auth3]{Henrik Sandberg}\ead{hsan@kth.se}  

\address[Auth1]{Department of Automatic Control, Lund University, P.O. Box 118, SE-221 00 Lund, Sweden }  
\address[Auth2]{Department of Mechanical Engineering at the University of California at Santa Barbara, Santa Barbara, CA 93106, USA}        
\address[Auth3]{School of Electrical Engineering and Computer Science, KTH Royal Institute of Technology, SE-100 44 Stockholm, Sweden }             

\begin{abstract}
We consider distributed consensus in networks where the agents have integrator dynamics of order two or higher ($n\ge 2$). We assume all feedback to be localized in the sense that each agent has a bounded number of neighbors and consider a scaling of the network through the addition of agents {in a modular manner, i.e., without re-tuning controller gains upon addition}. We show that standard consensus algorithms, {which rely on relative state feedback, }are subject to what we term scale fragilities, meaning that stability is lost as the network scales. 
For high-order agents ($n\ge 3$), we prove that no consensus algorithm with fixed gains can achieve consensus in networks of any size. That is, while a given algorithm may allow a small network to converge, it causes instability if the network grows beyond a certain finite size. This holds in families of network graphs whose algebraic connectivity, that is, the smallest non-zero Laplacian eigenvalue, is decreasing towards zero in network size (e.g. all planar graphs). 
For second-order consensus ($n = 2$) we prove that the same scale fragility applies to directed graphs that have a complex Laplacian eigenvalue approaching the origin (e.g. directed ring graphs). The proofs for both results rely on Routh-Hurwitz criteria for complex-valued polynomials and hold true for general directed network graphs. We survey classes of graphs subject to these scale fragilities, discuss their scaling constants, and finally prove that a sub-linear scaling of nodal neighborhoods can suffice to overcome the issue.  
\end{abstract}


\end{frontmatter}


\section{Introduction}
Characterizing the dynamic behaviors of networked or multi-agent systems has been an active research area for many years. In particular, since the works by 
\citet{FaxMurray}, 
\citet{OlfatiSaber2004}, and 
\citet{Jadbabaie2003}, the prototypical sub-problem of distributed consensus has been the subject of significant research efforts. While the particular modeling aspects vary, the consensus objective is to coordinate agents in a network to a common state of agreement. Applications range from distributed computing and sensing to power grid synchronization and coordination of unmanned vehicles~\citep{OlfatiSaber2007}.

The most traditional consensus problem is of first order, meaning that agents are modeled as single integrators with a state that develops according to a weighted sum of differences between states of neighboring agents, {that is, relative state feedback}. Second-order consensus can model coordination of agents with mass and is used to study formation control in multi-vehicle networks. The corresponding higher-order problem, to which most results in this paper pertain, has also received significant attention, as in~\citet{Ren2007, Ni2010, Rezaee2015,Jiang2009,Zuo2017,Radmanesh2017}. Here, each agent is modeled as an \nth order integrator, and the control signal is a weighted sum of {relative feedback terms. }
This can be viewed as an important theoretical generalization of the first- and second-order algorithms~\citep{Jiang2009}, but also has practical relevance. For example, position, velocity, as well as acceleration feedback play a role in flocking behaviors, resulting in a model where $n = 3$~\citep{Ren2007}.

Existing literature has typically focused on deriving conditions for convergence of a given set of agents to consensus, and how such conditions depend on various properties of the network. For example, directed communication, a switching or random topology~\citep{Ni2010}, or a leader-follower structure~\citep{Zuo2017}. This paper takes a different perspective and concerns the \emph{scalability} of given consensus algorithms to large networks {under a modular design principle}. That is, we assume that the interaction rules between agents are \emph{fixed}, {(i.e., pre-designed)} and \emph{localized}, and grow the network through the addition of more and more agents. It has previously been observed that this type of {modular }scaling can lead to poor dynamic behaviors in first- and second-order consensus problems, such as a lack of network coherence~\citep{Bamieh2012,Patterson2014, SiamiMotee2015,Tegling2019}. These behaviors are a question of control \emph{performance}. In this paper, we show that the question of scalability in high-order consensus is more fundamental: can \emph{stability} be maintained as the network grows?

This paper shows that both second- and higher-order consensus ($n\ge 2$) are subject to \emph{scale fragilities} in certain classes of network graphs. These imply that stability (and thereby convergence to consensus) is lost if the network grows beyond some finite size. For $n\ge 3$, our result is particularly clear-cut: the consensus algorithm treated in, for example, \citet{Ren2007} does not scale stably in any family of graphs whose algebraic connectivity decreases towards zero in network size. 

The algebraic connectivity, that is, the smallest non-zero eigenvalue of the graph Laplacian, decreases towards zero in families of graphs where nodal neighborhoods are \emph{localized} in the sense that they are bounded in size and reach (the formal definition is given through the graph's isoperimetric, or Cheeger, constant). Here, we review this property for graphs such as lattices, trees, and planar graphs, and derive the rates at which their respective algebraic connectivity decreases. 
In leader-follower consensus of order $n\ge 3$, the scale fragility arises in any undirected graph family where the neighborhood size is bounded. This latter result was observed in the context of vehicular strings by~\citet{Swaroop2006} and \citet{Barooah2005}. Here, we generalize that result to leaderless consensus and general directed, weighted graphs. 

For second-order consensus ($n = 2$), the scale fragility applies only to particular classes of directed graphs. These are characterized by a complex Laplacian eigenvalue that approaches the origin {as the network size grows}. This applies, for example, to directed ring graphs. The particular result for ring graphs has previously been reported in~\citet{Cantos2016,HermanThesis,Stuedli2017}, but our work provides a significant generalization.  The result implies that ring-shaped vehicular formations, such as those where adaptive cruise control is used to regulate spacing and velocity to the preceding vehicle, see~\citet{Gunter2021}, 
are at risk of becoming unstable.

{We remark that the phenomenon we describe in this paper is distinct from the issue of string stability in vehicular strings. String instability, see e.g.~\citet{Seiler2004, Swaroop2006}, implies that disturbances are amplified along the string of vehicles, though the overall system dynamics can be stable. It is therefore a notion of performance rather than stability, see also~\citet{Besselink2018}. Here, we describe a loss of closed-loop stability, subject to a {modular }scaling of the network.  }

The fact that consensus may fail to scale stably to large networks has, to the best of our knowledge, not been observed in literature apart from the aforementioned works. While it is noted in~\citet{Ren2007, Jiang2009} that controller gains in high-order consensus must be chosen with care to ensure stability, we point out that {no such choice can guarantee stability in a network that grows. For so-called \emph{open} multi-agent systems~\citep{Hendrickx2017,Franceschelli2021}, where agents may come and leave while adhering to, e.g., a consensus protocol, our results imply limitations on the allowable size of the overall system (depending on the agent dynamics and the degree of locality).  }


{The scale fragilities we describe can in principle be attributed to the relative state feedback upon which the consensus algorithm is based. It is known that a restriction to relative feedback imposes performance and design limitations; an issue that was analyzed formally in~\citet{Jensen2022}. In this paper, we also discuss how the scalability can be achieved if the controller has access to absolute feedback.   }

The locality property, that is, bounded nodal neighborhoods, is also key for our results. A natural question is therefore how nodal neighborhoods would need to scale to alleviate the scale fragility. Interestingly, we prove using a {ring graph} topology that it can suffice to grow neighborhoods as $N^{2/3}$, where $N$ is the network size. We note that this only holds for leaderless consensus; leader-follower consensus requires neighborhoods proportional to $N$. 

The present paper extends our preliminary work in~\citet{Tegling2019ACC}, where the result on high-order ($n\ge3$) consensus was first reported. The corresponding result herein is improved in its formalism and generalized to all directed graphs families. Our characterization of graphs with decreasing algebraic connectivity has been expanded with a general analytic criterion. All other results are new. 

The remainder of this paper is organized as follows. We next introduce the \nth order consensus algorithm along with important definitions and assumptions. In Section~\ref{sec:mainresult} we give the result for high-order consensus. We also discuss classes of graphs where the result applies and give numerical examples. Section~\ref{sec:2ndorder} then presents corresponding results for second-order consensus. In Section~\ref{sec:mitigation} we discuss ways to retrieve {scalable stability}, e.g. by scaling nodal neighborhoods, and we conclude with a discussion in Section~\ref{sec:discussion}. 

\section{Problem setup}
\label{sec:setup}
We now introduce the network model along with the \nth order consensus algorithm. This algorithm is a straightforward extension to standard first- and second-order consensus and has previously been considered in~\citet{Ren2007,Ni2010, Rezaee2015}. 

\subsection{Network model and definitions}
Consider a network modeled by the graph $\Ggn = \{\mathcal{V}_N, \mathcal{E}_N\}$ with $N = |\mathcal{V}_N|$ nodes. The set ${\mathcal{E}_N\subset \mathcal{V}_N\times \mathcal{V}_N}$ contains the edges, each of which has an associated nonnegative weight $w_{ij}$.
We generally let the graph $\Ggn$ be directed, so the edge~$(i,j) \in \Ee_N$ points from node~$i$ (the tail) to node~$j$ (the head).  The neighbor set~$\mathcal{N}_i$ of node~$i$ is the set of nodes~$j$ to which there is an edge ${(i,j)\in \mathcal{E}}$. 
The outdegree of node~$i$ is defined as ${d^+_i = \sum_{j =1}^N w_{ij}}$ and its indegree is $d^-_i = \sum_{j =1}^N w_{ji}$ (${w_{ij}=0}$ if $(i,j)\notin \mathcal{E}$). 
The graph~$\Ggn$ is \emph{balanced} if ${d^+_i = d^-_i}$ for all $i \in \mathcal{V}_N$ and  \emph{undirected} if $(i,j) \in \mathcal{E}_N\Rightarrow (j,i) \in \mathcal{E}$ for all $i,j \in \mathcal{V}_N$ and $w_{ij} = w_{ji}$. It 
has a \emph{connected spanning tree} if there is a path from some node $i\in \mathcal{V}_N$ to any other node~$j \in \mathcal{V}_N\backslash\{i\}$. {The $r$-fuzz of a graph $\Ggn$ is the graph obtained from $\Ggn$ by adding an edge $(u,v)$ for all $v$ that are at most $r$ steps away from $u$. }

Going forward, we will model networks with an increasing numbers of agents. 
We therefore consider $\Ggn$ as a member of a sequence, or a \emph{family}, of graphs~$\{\Ggn\}$ in which the network size $N$ is increasing. We remark that $\Gg_{N}$ need not be a subgraph of $\Gg_{N+1}$ for our results to hold. 

The graph Laplacian~$L$ of~$\Ggn$ is defined as follows:
\vspace{-0.5mm}
\begin{equation}
\label{eq:laplaciandef}
[L]_{ij} = \begin{cases}  -w_{ij}& \mathrm{if}~ j \neq i~\mathrm{and}~j \in \mathcal{N}_i
\\ 
\sum_{k \in \mathcal{N}_i} w_{ik} & \mathrm{if}~ j = i\\
0 & \mathrm{otherwise.}
 \end{cases}
\end{equation}
Denote by $\lambda_l$ (or $\lambda_l(\Ggn)$ where explicitness is needed) with $l = 1,\ldots,N$ the eigenvalues of~$L$. Zero is a simple eigenvalue of~$L$ if and only if the graph has a connected spanning tree, {which will be the scenario of interest throughout. }
Remaining eigenvalues are in the complex right half plane~(RHP), and numbered so that $0 = \lambda_1 < \mathrm{Re}\{\lambda_2\}\le \ldots \le \mathrm{Re}\{\lambda_N \}$. 
 The graph Laplacian~$L$ is called \emph{normal} if $L^TL = LL^T$. 
If the graph is undirected, $L$ is symmetric and thereby normal. For a directed graph, normality of $L$ implies that $\Ggn$ is balanced.

\subsection{\nth order consensus}
The local dynamics of each agent~$i\in \mathcal{V}_N$ is modeled as a chain of $n$ integrators: 
\begin{align*}
\ddt x_i^{(0)}(t)  &= x_i^{(1)}(t)\\
& ~\vdots \\
\ddt x_i^{(n-2)}(t)  &= x_i^{(n-1)}(t) \\ 
\ddt x_i^{(n-1)}(t)  &= u_i(t)
\end{align*}
where we assume a scalar state $x_i(t) \in \mathbb{R}$ (see Remark~\ref{rem:dimension}), {collected in the vector $x = [x_1,x_2,\ldots,x_N]^T\in \mathbb{R}^N$.} 
The notation for time derivatives is such that $x_i^{(0)}(t) = x_i(t)$, $x_i^{(1)}(t) = \ddt x_i(t)  = \dot{x}_i(t)$ etc. until $x_i^{(n)}(t) = \frac{\mathrm{d}^n}{\mathrm{d}t^n} x_i(t)$. 
Going forward, we will often drop the time dependence in the notation. 

We consider the following \nth order consensus algorithm:
\begin{equation}
\label{eq:consensuscompact}
u_i = - \sum_{k = 0}^{n-1} a_k\sum_{j \in \mathcal{N}_i} w_{ij}(x_i^{(k)} - x_j^{(k)})
\end{equation}
where the $a_k>0$ are fixed gains.
The feedback in~\eqref{eq:consensuscompact} is termed \emph{relative} as it only based on differences between states of neighboring agents. 
The impact of absolute feedback, where the controllers have access to measurements of the absolute local state, is treated in Section~\ref{sec:mitigation}.

Defining the full state vector ${\xi = [x^{(0)},x^{(1)},\ldots,x^{(n-1)}]^T \in \mathbb{R}^{Nn}}$, we can write the system's closed-loop dynamics as
\begin{equation}
\label{eq:closedloop}
\ddt \xi = \underbrace{\begin{bmatrix}
0 & I_N & 0 & \cdots & 0\\
0 & 0 & I_N & \cdots &\vdots \\
0 & 0 & 0 & \ddots &\vdots \\
0 & 0 & 0 & \cdots &I_N \\
-a_0L & -a_1L & -a_2L & \cdots & -a_{\mathrm{n-1}}L
\end{bmatrix}}_{\mathcal{A}} \xi,
\end{equation}
where the graph Laplacian~$L$ was defined in~\eqref{eq:laplaciandef} and $I_N$ denotes the $N\times N$ identity matrix.

\begin{rem} \label{rem:dimension}
We limit the analysis to a scalar information state, though an extension to $x_i(t) \in \mathbb{R}^m$ is straightforward {if the same consensus algorithm is applied in all coordinate directions}. In this case, the system dynamics can be written $\dot{\xi} =   (\mathcal{A} \otimes I_m) \xi$, where $\otimes$ denotes the Kronecker product. 
Our results, which concern the stability of~$\mathcal{A}$, would not change.
\end{rem}

\subsubsection{Leader-follower consensus}
\label{sec:LF}
It will also be relevant to consider leader-follower consensus, where the state of one agent (the leader) is fixed at a desired setpoint and remaining agents converge to that same state (assuming there is a directed path to each of them from the leader node). Without loss of generality, take Agent 1 to be the leader and set $x_1 = \dot{x}_1 = \ldots =x_1^{(n)} \equiv 0$. The closed-loop dynamics for remaining agents can then be written 
\begin{equation}
\label{eq:closedloopred}
\ddt \bar{\xi}= \underbrace{\begin{bmatrix}
0 & I_{N-1} & 0 & \cdots & 0\\
0 & 0 & I_{N-1} & \cdots &\vdots \\
0 & 0 & 0 & \ddots &\vdots \\
0 & 0 & 0 & \cdots & I_{N-1} \\
-a_0\bar{L} & -a_1\bar{L}  & -a_2\bar{L}  & \cdots & -a_{n-1}\bar{L} \\
\end{bmatrix}}_{\bar{\mathcal{A}}} \bar{\xi} ,
\end{equation}
where $\bar{L}$ is the \emph{grounded} graph Laplacian obtained by deleting the first row and column of $L$, and~$\bar{\xi}$ is obtained by removing the states of Agent~1. Note that~$\bar{L}$ unlike~$L$ has all of its eigenvalues in the right half plane~\citep{Xia2016}. 

\subsection{Conditions for consensus and {scalable stability}}
\label{sec:reachconsensus}
{The network of agents is said to be achieving consensus} if $x_i^{(k)} \rightarrow x_j^{(k)}$ for all $i,j \in \mathcal{V}_N$, all $k = 0,1,\ldots,n-1$, {and for any initial state}. It is known that the algorithm~\eqref{eq:consensuscompact} achieves consensus if the eigenvalues of~$\mathcal{A}$ are in the left half plane, apart from exactly $n$ zero eigenvalues that are associated with the drift of the network average. This condition is in line with standard results for first- and second-order consensus, and is shown in~\citet{Ren2007} for $n = 3$:
\begin{lemma}[\cite{Ren2007}, Theorem 3.1]
\label{thm:consensus}
In the case of $n = 3$, the algorithm~\eqref{eq:consensuscompact} achieves consensus exponentially if and only if $\mathcal{A}$ has exactly three zero eigenvalues and all of the other eigenvalues have negative real parts. 
\end{lemma}
We also require the following lemma:
\begin{lemma}[\citet{Ren2007}, Lemma 3.1]
In the case of $n = 3$, the matrix $\mathcal{A}$ has exactly three zero eigenvalues if and only if $L$ has a simple zero eigenvalue.
\end{lemma}
The proofs in~\citet{Ren2007} extend straightforwardly to $n>3$.
This means that it suffices to verify that the ${(N-1)\cdot n}$ non-zero eigenvalues of~$\mathcal{A}$ have negative real parts. 

In this work, we describe systems where these conditions may hold for small network sizes~$N$, but where one or more eigenvalues leaves the left half plane and causes instability when the network grows beyond some~$\bar{N}$. In these cases, we say the control algorithm lacks \emph{{scalable stability}}. 

\begin{definition}[{Scalable stability}]
A consensus control design is {\emph{scalably stable}} if the resulting closed-loop system achieves consensus over \emph{any} graph in the family~$\{\Gg_N\}$. 
\end{definition}

\subsection{Underlying assumptions: {modularity and locality}}
{The notion of scalable stability of a controller presumes a \emph{modular} design principle. This means that new agents are added to the network with the pre-designed controller gains, which are not re-tuned as the network grows. This means that the following important assumptions will underlie our analysis of the control $u$ in~\eqref{eq:consensuscompact}: }
\begin{ass}[Fixed {and finite} gains]
\label{ass:fixed}
The gains $a_k$ for all $k = 0,1,\ldots, n-1$ { satisfy $a_k\le a_{\max} <\infty $ and} they do not change if the underlying graph changes. That is, the gains are fixed with respect to the graph family~$\{\Ggn\}$. In particular, they are independent of~$N$.
\end{ass}

When it comes to the network graph, 
our main result will rely on the property that the algebraic connectivity decreases in network size. When discussing families of graphs where this property holds, we will impose the following assumptions, unless otherwise stated:
\begin{ass}[Bounded neighborhoods]
\label{ass:q}
All nodes in the graph family~$\{\Ggn\}$ have a neighborhood of size at most~$q$, where~$q$ is fixed and independent of~$N$. That is, 
\vspace{-1.5mm}
\begin{equation}
\label{eq:qdef}
 |\mathcal{N}_i| \le q~~\forall i \in \mathcal{V}_N. \vspace{0.8mm}
\end{equation} 
\end{ass}
\begin{ass}[Finite weights]
\label{ass:boundedweights}
The edge weights in~$\{\Ggn\}$  are finite, that is, $w_{ij}\le w_{\max} <\infty $ for all $(i,j)\in\Ee_N$, where $w_{\max}$ is fixed and independent of~$N$.
\end{ass}
Assumptions~\ref{ass:q}--\ref{ass:boundedweights} imply that we consider networks with bounded nodal degrees. 

\section{Scale fragility in high-order consensus}
\label{sec:mainresult}
This section is devoted to our first important result. We prove that the high-order consensus algorithm ($n\ge 3$) lacks {scalable stability} in graph families with what we term a {decreasing algebraic connectivity}. This applies to all graphs where connections are, in a sense, localized. 

\subsection{Main result}
This section's main result can be stated as follows. 
\begin{theorem}
\label{thm:main}
If $n\ge 3$, no control on the form~\eqref{eq:consensuscompact} subject to 
Assumption~\ref{ass:fixed}, 
is {scalably stable} in graph families where {the sequence} 
$\mathrm{Re}\{\lambda_2(\Ggn)\}\rightarrow 0$ as $N \rightarrow \infty$. 

\end{theorem}

\begin{proof} 
The first step of the proof is a (generalized) block-diagonalization of the system matrix~$\mathcal{A}$. Let $T$ be an invertible $N\times N$ matrix such that $\Lambda = T^{-1}LT$ is on Jordan normal form. That is, $\Lambda = \mathrm{diag}\{\Lambda_1, \cdots, \Lambda_k\}$, where $\Lambda_l$, $l = 1,\ldots, k$ are $r_l\times r_l$ Jordan blocks, in which the Laplacian eigenvalue $\lambda_l$ is repeated along the main diagonal and ones appear on the superdiagonal (see~\citet[Chapter 3]{HornJohnson} for details). The number~$k$ of Jordan blocks is the number of linearly independent eigenvectors of $L$, which may be less than or equal to its number of distinct eigenvalues. If the graph~$\Gg_N$ is undirected, then $L$ is symmetric and thus diagonalizable. In this case, $r_i = 1$ for $i = 1,\ldots,k = N$. Otherwise, we only impose that the eigenvalue $\lambda_1 = 0$ is simple, which is equivalent to the graph having a connected spanning tree. {If this is not the case, the graph is disconnected, $\lambda_2(\Ggn) = 0$, and the conditions in Section~\ref{sec:reachconsensus} do not hold. The system is then by definition not scalably stable. }
By pre- and post-multiplying $\mathcal{A}$ by the $(Nn \times Nn)$ matrix $\mathbf{T} = \mathrm{diag} \{T,T,\ldots,T\}$, we get 
\begin{equation}
\label{eq:diagonalizing}
\mathbf{T}^{-1}\mathcal{A}\mathbf{T} = \underbrace{\begin{bmatrix}
0 & I_{N} & 0 & \cdots & 0\\
0 & 0 & I_{N} & \cdots &\vdots \\
 & \vdots &  & \ddots &\vdots \\
0 & 0 & 0 & \cdots &I_{N} \\
-a_0\Lambda & -a_1 \Lambda & -a_2\Lambda & \cdots & -a_{\mathrm{n-1}}\Lambda
\end{bmatrix}}_{\hat{\mathcal{A}}}.
\end{equation}
{By pre- and post-multiplying by a suitable permutation matrix, 
the rows and columns of $\hat{\mathcal{A}}$ can be rearranged into the system matrix $\mathrm{diag}\{\hat{\mathcal{A}}_1,\ldots, \hat{\mathcal{A}}_k\}$ with}
\[ \hat{\mathcal{A}}_l = {\begin{bmatrix}
0 & I_{r_l} & 0 & \cdots & 0\\
0 & 0 & I_{r_l} & \cdots &\vdots \\
 & \vdots &  & \ddots &\vdots \\
0 & 0 & 0 & \cdots &I_{r_l} \\
-a_0\Lambda_l & -a_1 \Lambda_1 & -a_2\Lambda_l & \cdots & -a_{\mathrm{n-1}}\Lambda_l
\end{bmatrix}}\]
 for $l = 1,\ldots,k$. 
 The eigenvalues of~$\mathcal{A}$, {equivalently $\hat{\mathcal{A}}$,} are the union of the eigenvalues of all $\hat{\mathcal{A}}_l$ {since these are decoupled from each other}. Clearly, the~$n$ zero eigenvalues of $\Aa$ are obtained from $\hat{\mathcal{A}}_1$ since $\Lambda_1 = \lambda_1 = 0$. Therefore, to ensure {scalable stability}, we must require all eigenvalues of all $\hat{\mathcal{A}}_l$, $l = 2,\ldots,k$ to have negative real parts for any~$N$. 

The characteristic polynomial of each~$\hat{\mathcal{A}}_l$ is
\begin{equation}
\label{eq:charpoly}
P_l(s) = (\underbrace{s^n + a_{n-1} \lambda_l s^{n-1} + \ldots + {a}_1\lambda_ls + a_0\lambda_l}_{p_l(s)})^{r_l},
\end{equation}
whose roots are given by the roots of $p_l(s)$. 
In general, the eigenvalue $\lambda_l$ appearing in $p_l(s)$ is complex-valued. We therefore apply the Routh-Hurwitz criteria for polynomials with complex coefficients. As these criteria do not appear frequently in literature, we re-state {them in~\ref{RHapp}.} 

The first Routh-Hurwitz criterion applied to $p_l(s)$ reads 
\begin{equation}
\label{eq:trivialcrit}
a_{n-1}\mathrm{Re}\{\lambda_l\}>0.
\end{equation}
Since $\mathrm{Re}\{\lambda_l\}>0$ for $l = 2,\ldots,k$ this is always satisfied when $a_{n-1}>0$. The second criterion, {given in~\eqref{eq:secondcond},} can after some manipulation be written as 
\begin{multline}
\label{eq:RHcriterion}
a_{n-1}(\rell)^2(a_{n-1}a_{n-2}\rell - a_{n-3})  \\ + a_{n-2}(\imll)^2(a_{n-1}^2\rell - a_{n-2}) >0,
\end{multline} 
which must hold for all $l = 2,\ldots,k$.
While the factors in front of the brackets remain positive for all~$\lambda_l$ (recall, $a_k>0$), the brackets themselves are negative if $\rell$ is sufficiently small. In particular, the condition~\eqref{eq:RHcriterion} is violated if $\rel = \min_l \rell$ is sufficiently small. 

This means that if the criterion~\eqref{eq:RHcriterion} is evaluated for a graph family~$\{\Gg_N\}$ in which $\mathrm{Re}\{\lambda_2(\Ggn)\}\rightarrow 0$ as $N\rightarrow \infty$, it will eventually ({for a sufficiently large, but finite, $N$}) be violated. 

We can conclude that at least one root of the characteristic polynomial~$p_2(s)$ will have a nonnegative real part {for sufficiently large~$N$}. Lemma~\ref{thm:consensus} is then not satisfied and the control is not {scalably stable} for~${n\ge3}$.
\end{proof}
\begin{rem}
If the graph is undirected, then the polynomial~$\eqref{eq:charpoly}$ has real-valued coefficients. The result can then be derived using the standard Routh-Hurwitz criteria. This gives the simpler condition 
\begin{equation}
\label{eq:realRH}
a_{n-1}a_{n-2}\lambda_2(\Gg_N) - a_{n-3} >0,
\end{equation}
which cannot remain satisfied if $\{\lambda_2(\Ggn)\} \rightarrow 0$ as $N\rightarrow \infty$. 
\end{rem}

Theorem~\ref{thm:main} implies that high-order consensus does not scale in certain graph families. 
Instability will occur at the smallest size~$N$ for which the Routh-Hurwitz criteria are violated, and at least one eigenvalue crosses to the right half plane. We will denote this critical network size~$\bar{N}$.  
In Figure~\ref{fig:criticalN} we display~$\bar{N}$ for $n = 3,4,5$ in an unweighted path graph.

\subsubsection{High-order leader-follower consensus}
Leader-follower consensus~\eqref{eq:closedloopred} in undirected graphs lacks {scalable stability} under a weaker condition, namely, under bounded nodal degrees. 
This was also observed in~\citet{Swaroop2006}. 

We first require the following Lemma:
\begin{lemma}
\label{lem:leaderfollowerlemma}
Consider the grounded Laplacian matrix~$\bar{L}$ of an undirected graph~$\Ggn$ and let Assumptions~\ref{ass:q}--\ref{ass:boundedweights} hold. The smallest eigenvalue~$\bar{\lambda}_1$ of~$\bar{L}$ then satisfies
\begin{equation}
\label{eq:lambda1cond}
\bar{\lambda}_1(\Ggn) \le \frac{q}{N-1}w_{\max}.
\end{equation}
\end{lemma}
\begin{proof}
By the Rayleigh-Ritz theorem~\citep[Theorem 4.2.2]{HornJohnson} it holds
\begin{equation*}
\bar{\lambda}_1 \le  \frac{v^*\bar{L}v}{v^*v} ,~~ \forall v \in \mathbb{C}^{N-1} \backslash \{0\}.
\end{equation*}
This implies in particular that 
\[ \bar{\lambda}_1 \le  \frac{\mathbf{1}_{N-1}^T\bar{L}\mathbf{1}_{N-1}}{\mathbf{1}_{N-1}^T\mathbf{1}_{N-1}} = \frac{\sum_{k \in \mathcal{N}_1} w_{1k}}{N-1} \le \frac{q w_{\max}}{N-1},\]
where $\mathbf{1}_{N-1}^T\bar{L}\mathbf{1}_{N-1} = \sum_{k \in \mathcal{N}_1} w_{1k}$ is the weight sum of all edges leading to the leader node 1.  The equality holds since each row~$k$ of the grounded Laplacian~$\bar{L}$ sums to zero if the corresponding node~$k$ has no connection to the leader, and otherwise to~$w_{1k}\le w_{\max}$.
\end{proof}
Clearly, $\bar{\lambda}_1(\Ggn) \rightarrow 0$ as $N\rightarrow \infty$. The next theorem therefore follows. 
\begin{theorem}
\label{thm:leaderfollower}
 If $n\ge 3$, no leader-follower consensus algorithm on the form~\eqref{eq:closedloopred} is {scalably stable} in undirected graph families~$\{\Ggn\}$ under Assumptions~\ref{ass:fixed}--\ref{ass:boundedweights}.
\end{theorem}
\begin{proof}
The arguments in the  proof of Theorem~\ref{thm:main} apply. In this case,  
$N-1$ real-valued characteristic polynomials $p_l(s)$ as in~\eqref{eq:charpoly} are obtained. 
We can use the condition~\eqref{eq:realRH}, which in this case reads
$a_{n-1}a_{n-2}\bar{\lambda}_l - a_{n-3} >0$ 
for $l = 1,\ldots,N-1$. 
By Lemma~\ref{lem:leaderfollowerlemma}, that requires
\begin{equation}
\label{eq:reducedcond}
a_{n-1}a_{n-2} > \frac{1}{q w_{\max}}a_{n-3}(N-1),
\end{equation}
which will be violated for sufficiently large $N$, preventing {scalable stability}. 
\end{proof}
\begin{rem}
Assumption~\ref{ass:q} of bounded neighborhoods can be relaxed. As seen from~\eqref{eq:reducedcond}, Theorem~\ref{thm:leaderfollower} 
holds if $q/N\rightarrow 0$ as $N\rightarrow \infty$. That is, if nodal neighborhoods have sublinear growth in~$N$. 
\end{rem}

\begin{figure}
\centering
\begin{tikzpicture}
	\begin{axis}
	[xlabel={Neighborhood size $q$ },
	ylabel={Critical network size $\bar{N}$ },
	ylabel near ticks,
	xlabel near ticks,
	xmin=1,
	xmax=40,
	ymin=0,
	ymax=500,
	yticklabel style={
        /pgf/number format/fixed,
        /pgf/number format/precision=5
},
	grid=major,
		height=4.420cm,
	width=0.98\columnwidth,
	legend cell align=left,
	legend style={at={(0,1)},anchor= north west, fill = white}, font = \small,
	legend entries={$n = 3 $\\ $n = 4$ \\  $n = 5$\\}, cycle list name=mycolorlist3,
	]
			\foreach \x in {1,2,3}{
	\addplot +[mark=star, thick, only marks, ] table[x index=0,y index=\x,col sep=space]{consensusdata2.txt};
	},
	\foreach \x in {1,2,3}{
	\addplot +[thick] table[x index=0,y index=\x,col sep=space]{scaledata3.txt};
	}
	\end{axis}
	\end{tikzpicture}	
\caption{Critical network size~$\bar{N}$ at which the stability conditions are violated for an \nth order consensus algorithm. The graph is an undirected path graph where each node is connected to its $q$ nearest neighbors. Increasing the neighborhood size~$q$ here increases~$\bar{N}$ faster than linearly {-- Theorem~\ref{thm:qscale} predicts $\bar{N} = \mathcal{O}(q^{3/2})$, indicated by dashed lines in the plot}. Also note that for higher model order~$n$, the stability conditions are violated at smaller~$\bar{N}$. 
 }
\label{fig:criticalN}
\end{figure}
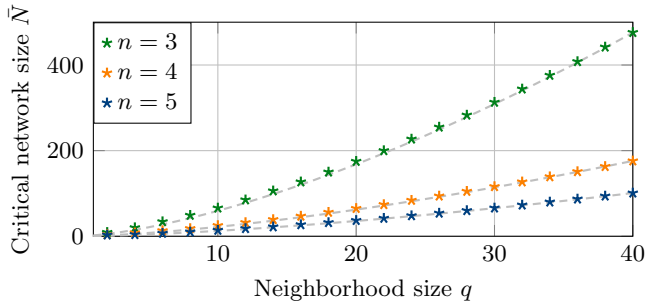

\subsection{Affected classes of graphs}
\label{sec:graphs}
We proved that high-order consensus lacks {scalable stability} in any network where the underlying graph family is such that $\mathrm{Re}\{\lambda_2(\Ggn)\}$ is decreasing towards zero as~$N$ increases. For undirected graphs, the smallest non-zero Laplacian eigenvalue~$\lambda_2$ is real-valued and known as the \emph{algebraic connectivity} of the graph. For directed graphs, the notion of algebraic connectivity is not clear-cut, see e.g.~\citet{Chung2005}. We can, however, make the following statement:
\begin{lemma}
\label{lem:mirrorlemma}
If~$L$ is {normal}, then
$\rel = \lambda_2^s,$
where $\lambda_2^s$ is the smallest non-zero eigenvalue of $L^s = {(L + L^T)/2}$, that is, the symmetric part of~$L$. 
\end{lemma}
\begin{proof}
With $v$ an eigenvector, $Lv = \lambda_2v$, and since $L$ is normal $L^Tv = \lambda_2^*v$, where $^*$ denotes complex conjugate. Then, $\frac{1}{2}(L \!+\!L^T)v = \frac{1}{2}(\lambda_2\! +\! \lambda_2^*)v = \frac{1}{2}( 2\rel) v$.
\end{proof}
For any balanced graph, the matrix~$L^s$ is the graph Laplacian corresponding to the \emph{mirror graph} $\hat{\mathcal{G}}_N$ of $\mathcal{G}_N$. {The mirror graph (of any directed graph)} is the {undirected} graph obtained as $\hat{\mathcal{G}}_N = \{\mathcal{V}_N, \mathcal{E}_N\cup \hat{\mathcal{E}}_N\}$, where $\hat{\mathcal{E}}_N$ is the set of all edges in $\mathcal{E}_N$, but \emph{reversed}, and whose edge weights are $\hat{w}_{ij} = \hat{w}_{ji} = (w_{ij}\!+\!w_{ji})/2$~\citep{OlfatiSaber2004}. Clearly, the mirror graph of an undirected graph is the graph itself. 
Lemma~\ref{lem:mirrorlemma} implies that when $L$ is normal, $\mathrm{Re}\{\lambda_2(\Ggn)\}$ is obtained as the algebraic connectivity of the mirror graph~$\mathcal{\hat{G}}_N$. 

We conclude that the result in Theorem~\ref{thm:main} will apply to \emph{graph families whose Laplacians are normal and where the corresponding mirror graph family has a decreasing algebraic connectivity}. That is, where $\{\lambda_2(\hat{\Gg}_N)\}\rightarrow 0$ as $N\rightarrow \infty$.
It is therefore meaningful to identify this property in {undirected} graph {families, which is what the remainder of this section is devoted to.}
 We first {state a general condition, and then }survey particular classes of graphs. 

\begin{rem}
\label{rem:nonnormal}
For directed graph families with non-normal Laplacians, a conclusion regarding the sequence  $\mathrm{Re}\{\lambda_2(\Ggn)\}$ cannot in general be drawn from the mirror graphs. A notable counter-example is the directed path graph on $N$ nodes with the edge set $\Ee_N = \{(i,i+1)~|~ i = 1,\ldots,N-1\}$. Here, $\mathrm{Re}\{\lambda_2(\Ggn)\} = 1$ for any $N$, while ${\lambda_2(\hat{\Gg}_N) = 1 - \cos\frac{\pi}{N}}$.  For general directed graphs, the sequence $\mathrm{Re}\{\lambda_2(\Ggn)\}$  must therefore be checked case by case. 
\end{rem}

\subsubsection{{Condition on the Cheeger constant}}
In general, the algebraic connectivity decreases in~$N$ in any undirected graph family that is not an \emph{expander family}. To define expander families, we require the Cheeger constant (also called isoperimetric constant), which for non-regular weighted graphs can be defined as~\citep[Chapter 2]{ChungBook}:
\begin{equation}
h(\Gg) = \inf_{X \subset \Vv} \frac{ |\partial X|_d}{\min\{|X|_d, |\bar{X}|_d\} } .
\label{eq:cheeger}
\end{equation}
Here, $\bar{X} = \Vv \backslash X$ and $\partial X= \{j \in \bar{X} ~|~ (i,j) \in \Ee,~i \in X \}$ is  {called the }the boundary set of $X$. Sets of nodes are measured here as $|W|_d := \sum_{i \in W}d_i$, where the nodal degree $d_i = \sum_{j \in \mathcal{N}_i} w_{ij}$. Loosely speaking, a large Cheeger constant implies that any subset of nodes is well connected to the rest of the graph, and it is not possible to find a ``bottleneck'' that separates two graph partitions from each other as they grow. See~\citet{Tegling2019NECSYS} for an elaboration and an algebraic condition. 
Now, consider the following definition. 
\begin{definition}[Expander family]
Let $\{\Gg_N\}$ be a graph family in which $N\rightarrow \infty$. 
If the sequence $\{h(\Gg_N)\}$ is bounded away from zero, $\{ \Gg_N\}$ is an expander family. 
\end{definition}

The following well-established result relates expander families to our problem:
\begin{lemma}
\label{res:expanders}
The sequence $\{\lambda_2(\Gg_N)\}$ is bounded away from zero as $N\rightarrow \infty$ if and only if $\{ \Gg_N\}$ is an expander family.
\end{lemma}
See e.g. \citet[Chapter 1]{KrebsBook} for a proof. Lemma~\ref{res:expanders} implies that a bounded-degree graph family can have an algebraic connectivity that does not decrease towards zero, if the same holds for the Cheeger constant. The equation~\eqref{eq:cheeger} reveals that this requires edges to connect across the entire network. In other words, that feedback is \emph{non-localized}. 

{Expander graphs with bounded degrees are difficult to construct explicitly, but they may arise through random processes. For example, the regular random graph family constructed by assigning edges through equally likely permutations of the node set~$\Vv_N$, will \emph{almost surely} be an expander family~\citep{Friedman1991}. }

{Next, we turn our attention to typical graph families that are non-expanding and thus have a decreasing algebraic connectivity.   }

\begin{rem}
It is noteworthy that Theorem~\ref{thm:leaderfollower} for leader-follower consensus applies even though $\{\Gg_N\}$ is an expander family. This means that leaderless consensus, despite being {scalably stable} in expander graphs, will be destabilized if one agent becomes a leader ("is grounded"). This fragility is described in detail in~\citet{Tegling2019NECSYS}.
\end{rem}

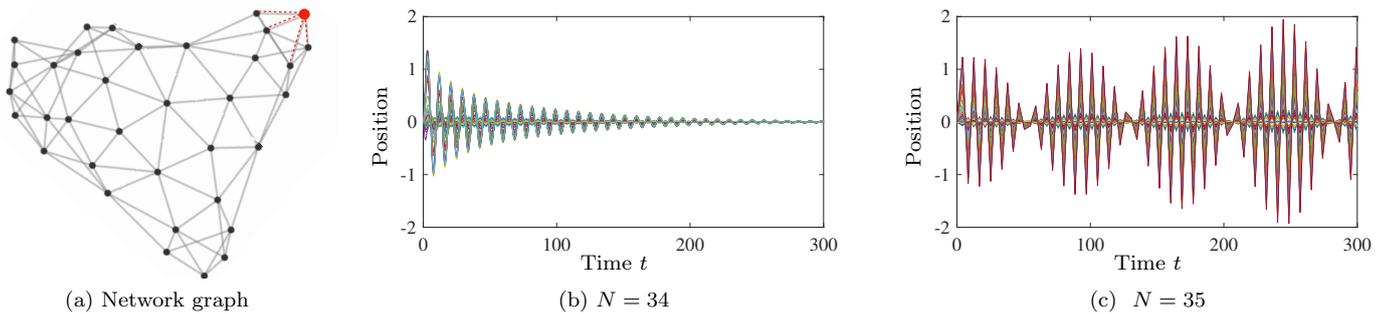
\begin{figure*}[htb]
  \centering
  \subfloat[][{Network graph}]{
  \includegraphics[width = 0.26\textwidth]{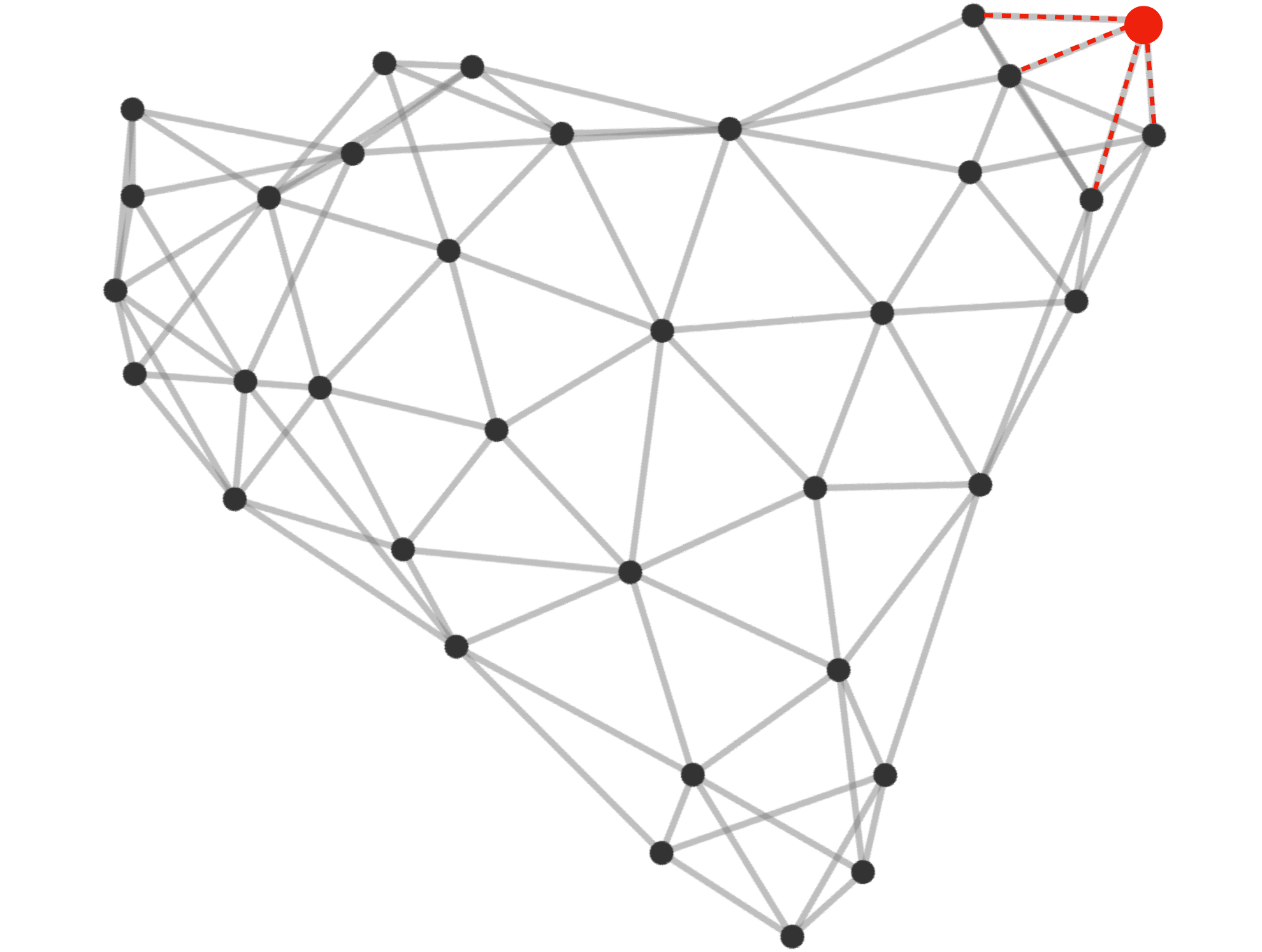}
  \label{fig:graph}  }
  \subfloat[][{$N=34$}] {
  \begin{tikzpicture}
		\node[inner sep=0pt] (img) at (0,0)
		{ \includegraphics[width = 0.37\textwidth]{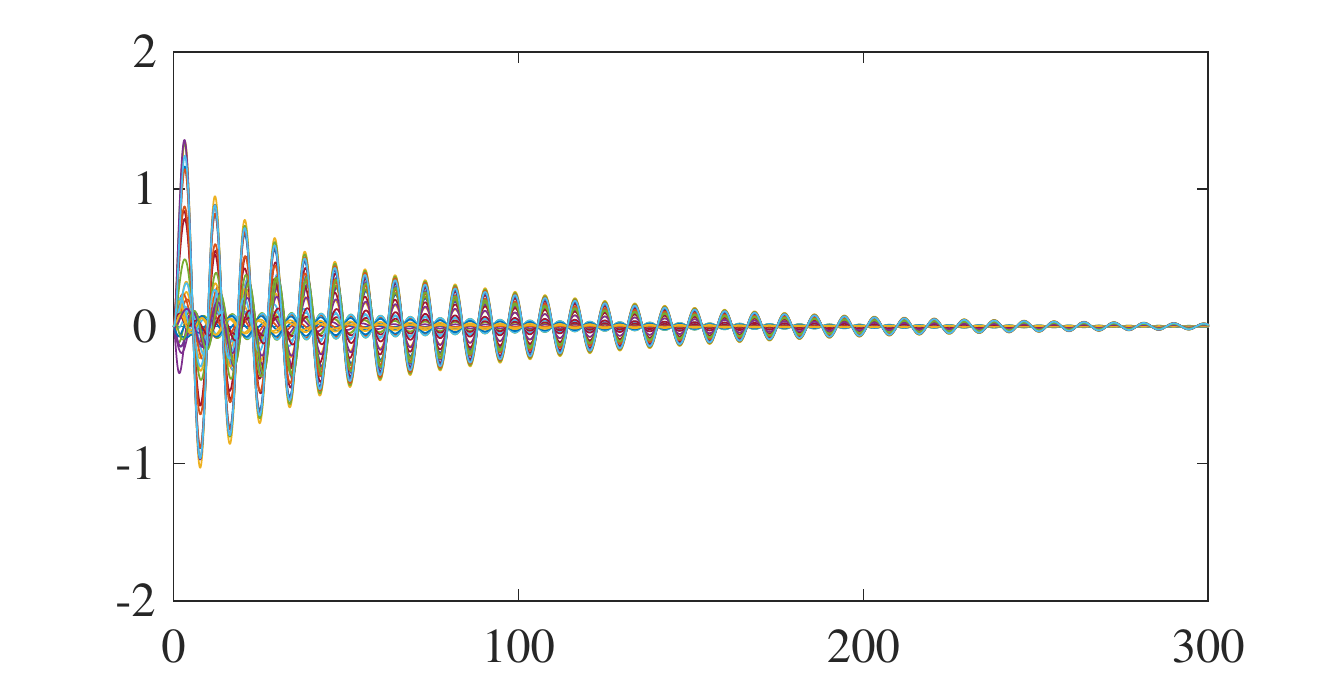}};
		\draw (0, -1.8) node {{\footnotesize Time $t$}};
		\node[xshift=-3.3cm,rotate=90,anchor=north]{{\footnotesize Position}};
	\end{tikzpicture}
    \label{fig:34node}}
   \subfloat[][{ $N=35$}] {
    \begin{tikzpicture}
		\node[inner sep=0pt] (img) at (0,0)
		{  \includegraphics[width = 0.37\textwidth]{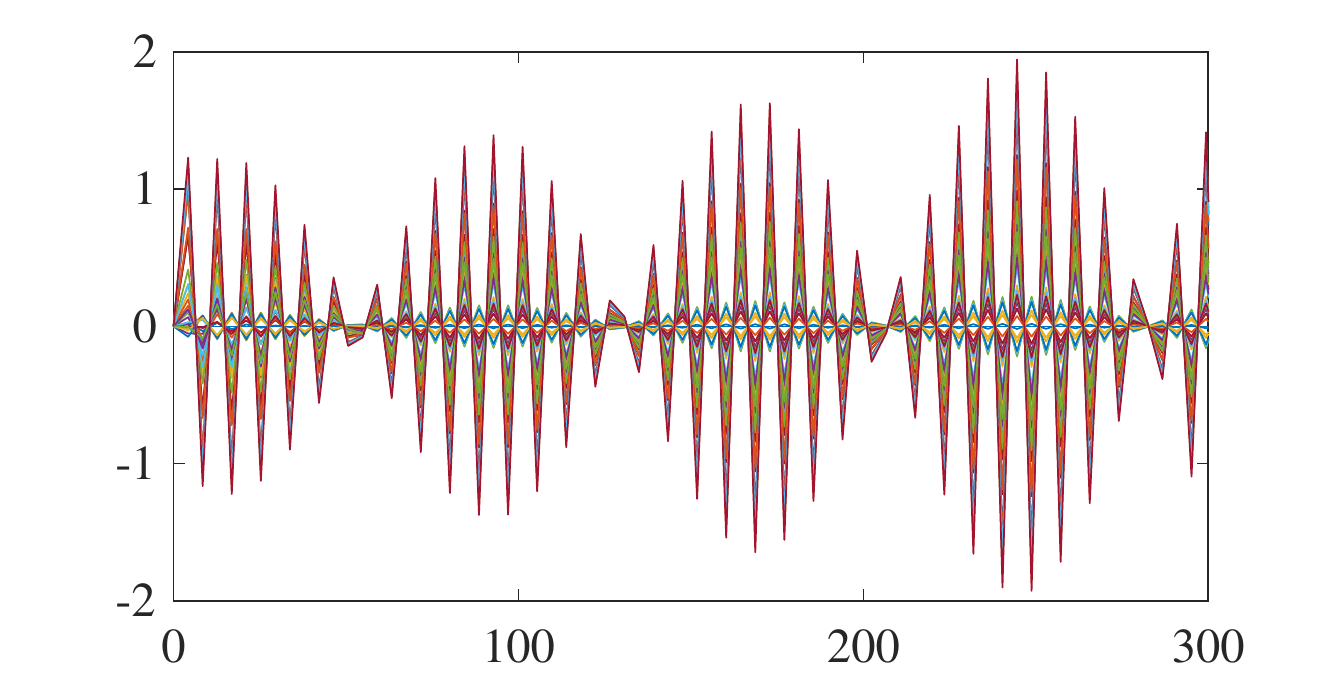}};
		\draw (0, -1.8) node {{\footnotesize Time $t$}};
		\node[xshift=-3.3cm,rotate=90,anchor=north]{{\footnotesize Position}};
	\end{tikzpicture}
   \label{fig:35node}} \\
       \caption{
{Simulation of $3^{\mathrm{rd}}$ order consensus over graph depicted in (a) subject to random initial accelerations. In (b) the network's 34 agents converge to an equilibrium. In (c) a $35^{\mathrm{th}}$ node has been added, indicated by red color in the graph. This addition leads to instability. The plots (b) and (c) show position trajectories relative to Agent no. 1.   } }
\label{fig:simulation}
\end{figure*}

\subsubsection{Lattices, fuzzes and their embedded graphs}
\label{sec:lattices}
Consider a graph over the $d$-dimensional periodic lattice~$\mathbb{Z}_M^d$ with ${N = M^d}$ nodes, and let each node be connected to its $r$ neighbors in each lattice direction. {We term this graph, which is the Cartesian product of $d$ $r$-fuzzes of ring graphs, a $d$-dimensional \emph{$r$-fuzz lattice}. 
This graph is regular and the neighborhood size is $q = 2rd$. }
\begin{lemma}[Algebraic connectivity of $r$-fuzz lattices]
\label{lem:fuzz}
For {undirected $d$-dimensional $r$-fuzz lattices}
\begin{equation}
\label{eq:lambda2fuzz}
\lambda_2(\Ggn) = \mathcal{O}\left(\frac{1}{N^{2/d}}\right)
\end{equation}
\end{lemma}
\begin{proof}
{See~\citet{Tegling2019}. }
\end{proof}
The decay rate~\eqref{eq:lambda2fuzz} also holds for any subgraph of the $r$-fuzz lattice, that is, any graph that is \emph{embeddable} in it. In particular, lattices without periodic boundary conditions. This follows from the following important lemma:
\begin{lemma}
\label{lem:addedge}
Adding an edge to an {undirected} graph $\Ggn$, or increasing the weight of an edge, can only increase (or leave unchanged)~$\lambda_2(\Ggn)$, and vice versa.
\end{lemma}
\begin{proof}
\textit{Adding an edge: }See~\citet[Theorem 3.2]{Mohar1991}. 
\textit{Increasing edge weight: }If the weight of the edge $(i',j')$ is increased by $\Delta w$, the new graph Laplacian can be written $L' = L + \Delta L$, where $\Delta L$ is also a positive semidefinite graph Laplacian (of a disconnected graph). By~\citet[Theorem 2.8.1]{BrouwerBook} this implies that $\lambda_l' \ge \lambda_l$ for each $l = 1,\ldots,N$, and in particular~$\lambda_2' \ge \lambda_2$.
\end{proof}

\subsubsection{Planar graphs}
Planar graphs are embeddable in two-dimensional lattices, so Lemma~\ref{lem:fuzz} applies. For this important case, however, a more precise bound is available:
\begin{lemma}[Algebraic connectivity of planar graphs]
\label{lem:planar}
For {undirected }planar graphs,
\begin{equation}
\lambda_2(\Ggn) \le \frac{8qw_{\max}}{N},
\end{equation}
\end{lemma}
\begin{proof}
See \citet[Theorem 6]{Spielman2007}.
\end{proof}

\subsubsection{Tree graphs with growing diameter}
The \emph{diameter} $\mathrm{diam}\{\mathcal{G}\}$ of a graph~$\mathcal{G}$ is defined as the longest distance between any two nodes in the graph. If we let $\mathcal{G}$ be a \emph{tree graph}, then, by \citet[Corollary 4.4]{Grone1990} it holds $\lambda_2 \le 2w_{\max}\left(1- \cos\left(\frac{\pi}{\mathrm{diam}(\mathcal{G}) +1}\right) \right)$. 
This allows us to show the following lemma. 
\begin{lemma}[Algebraic connectivity of tree graphs]
For {undirected }tree graphs,
\begin{equation}
\label{eq:tree}
\lambda_2(\Ggn) \le \frac{\pi^2 w_{\max}}{(\mathrm{diam}(\mathcal{G}_N) +1 )^2}.
\end{equation}
\end{lemma}
\begin{proof}
Follows from the relation above and the fact that $1-\cos x \le \frac{x^2}{2}$ for any~$x$. 
\end{proof}
In our case, the tree diameter will always increase in~$N$ as a consequence of Assumption~\ref{ass:q}. Therefore,  $\{\lambda_2(\Ggn)\} \rightarrow 0$ as $N \rightarrow \infty$.

\subsection{Numerical examples}
\label{sec:numerical}
We next provide two numerical examples to illustrate the issue of {scalable stability} in high-order consensus. 

\subsubsection{Critical network size, locality and model order}
\label{sec:scaleexample}
Consider a family of undirected path graphs where each node is connected to its $q/2$ nearest neighbors in each direction (i.e., a $q/2$-fuzz {of a path graph}). For any given~$N$, the graph's connectivity is greater, the greater $q$ is. Increasing~$q$ thus delays the violation of the stability criteria in Theorem~\ref{thm:main}.

In Figure~\ref{fig:criticalN}, we depict the critical network size~$\bar{N}$ as a function of the neighborhood size~$q$. Here, we have selected a consensus algorithm where $a_0 = 0.1$, $a_1 = 0.8$, ${a_2,a_3,a_4 = 1}$, and all edge weights $w_{ij} = 1$. The plot shows that increasing~$q$ increases the critical network size, faster than linearly. In Section~\ref{sec:mitigation} we discuss the precise scaling of~$q$ in $N$ required to defer instability completely. 

We also note that as the model order~$n$ increases, the system becomes unstable at smaller~$\bar{N}$. This is because the higher-order Routh-Hurwitz conditions in~\eqref{eq:Tondlcrit} are violated before the lower-order ones. It is also in line with common control-theoretic intuition.

\subsubsection{Instability through node addition}
Our second example illustrates the phase transition -- from consensus to instability -- that the system experiences as the critical network size is reached. Figure~\ref{fig:graph} shows a planar graph that has been randomly generated by means of triangulation. Here, the maximum neighborhood size is $q = 8$ and the median is~5. All edge weights are set to 1. 

We consider a third-order consensus algorithm: \[x_i^{(3)} = -\sum_{j \in \mathcal{N}_i}\left[  0.5({x}_i - {x}_j)  + (\dot{x}_i - \dot{x}_j)+ (\ddot{x}_i - \ddot{x}_j) \right], \]
which by Lemma~\ref{thm:consensus} will achieve consensus if $\lambda_2>0.5$. With 34 nodes, the graph in Figure~\ref{fig:graph} has $\lambda_2(\Gg_{34}) = 0.536$ and the system achieves consensus, as seen from the simulation in Figure~\ref{fig:34node}. We then add a $35^{\mathrm{th}}$ node along with 4 connecting edges, as indicated in red color in the graph in Figure~\ref{fig:graph}. Now, $\lambda_2(\Gg_{35}) = 0.493$ and the system becomes unstable.\footnote{This particular value for $\lambda_2(\Gg_{35})$ depends on the placement of the $35^{\mathrm{th}}$ node. Other placements can allow the critial $\bar{N}>35$, but instability occurs eventually.  }
Figure~\ref{fig:35node} shows how the agents' positions~$x$ oscillate at an increasing amplitude.

\section{Scale fragility in second-order consensus}
\label{sec:2ndorder}
Next, we turn our attention to consensus in second-order integrator networks ($n = 2$). This case is particularly relevant as this model is used in formation control problems~\citep{OlfatiSaber2006}. {Scalable stability} is easily satisfied in second-order consensus if the underlying graph family is {undirected}\footnote{{This is evident from the upcoming proof of Theorem~\ref{thm:2ndorder}. Consider~\eqref{eq:secondorderchareqn} and note that if $\Ggn$ is undirected, $\lambda_l(\Ggn)$ are real-valued and positive. Since $a_1,a_0>0$, all roots of $p_l(s)$ are then in the left half plane for any~$N$.}} (though performance issues like string instability~\citep{Swaroop1996} and lack of coherence~\citep{Bamieh2012} may still be a concern).
We show here, however, that it fails to scale stably in certain families of \emph{directed} graphs {with complex eigenvalues}. {More precisely, graph families where the real part of one or more Laplacian eigenvalues approaches zero as $N$ grows and at least one of these eigenvalues has a relatively large imaginary part. The precise condition, which is illustrated in Figure~\ref{fig:evcond}, is stated in Theorem~\ref{thm:2ndorder}. First,} {we remind the reader that the Laplacian eigenvalues are ordered as $0 = \lambda_1(\Ggn) < \mathrm{Re}\{\lambda_2(\Ggn)\}\le \ldots \le \mathrm{Re}\{\lambda_N(\Ggn)\} $.}

\begin{theorem}
\label{thm:2ndorder}
{If $n\ge 2$, no control on the form~\eqref{eq:consensuscompact}, subject to Assumption~\ref{ass:fixed}, 
is {scalably stable} in graph families where, for a fixed index {$\bar{l} \in \{2,3\ldots,N\}$},}
\begin{enumerate}
\item {$\mathrm{Re}\{\lambda_{\bar{l}}(\Ggn)\}\rightarrow 0$ as $N \rightarrow \infty$, \emph{and}}
\item {for each $N$ and at least one $l \in \{2,3,..., \bar{l}\}$ it holds ${\mathrm{arg}\{\lambda_l(\Ggn)\}> {\psi} }$, where ${\psi \in (0,\pi/2)}$ is a constant angle independent of~$N$. }
\end{enumerate}
\end{theorem}

\begin{proof}
For $n\ge 3$ the result follows immediately from Theorem~\ref{thm:main} { (note, $\mathrm{Re}\{\lambda_{\bar{l}}(\Ggn)\}\rightarrow 0 \Rightarrow \mathrm{Re}\{\lambda_2(\Ggn)\}\rightarrow 0$)}. For $n = 2$, we proceed as in the proof of Theorem~\ref{thm:main} to obtain the characteristic polynomials
\begin{equation}
\label{eq:secondorderchareqn}
p_l(s) = s^2 + a_1\lambda_ls + a_0\lambda_l,
\end{equation} for $l = 2,\ldots,N$.
The Routh-Hurwitz criterion derived from $\Delta_4>0$ in~\eqref{eq:Tondlcrit} with $f_{n-3}=g_{n-3} =0$ becomes
\begin{align}
 \nonumber
 a_1^2 \rell [  (\rell)^2 +(\imll)^2  ] - a_0(\imll)^2  >0 .
\end{align}
If $\imll = 0$, this is clearly satisfied since $\rell>0$. For all $l \in \{2,\ldots,N\}$ where $\imll \neq 0$ we can re-formulate the condition as 
\begin{align}
\label{eq:conditionsecondorder} 
 a_1^2\rell \left[\left(\frac{\rell}{\imll}\right)^2 +1 \right] - a_0  >0 .
\end{align}
If the expression in brackets is upper bounded by some constant, this condition will be violated whenever $\rell$ is sufficiently small. 
{Therefore, if the condition \eqref{eq:conditionsecondorder} is evaluated for a graph family $\{\Gg_N\}$ in which there are eigenvalues for which $\mathrm{Re}\{\lambda_l(\Ggn)\}\rightarrow 0$  as $N\rightarrow \infty$, and it holds $\left(\frac{\mathrm{Re}\{\lambda_l(\Ggn)\}}{\mathrm{Im}\{\lambda_l(\Ggn)\}}\right)^2\le \mathrm{const.}$ for at least one of them, then the condition is eventually violated, and stability is lost.  In our case, $\mathrm{Re}\{\lambda_{\bar{l}}(\Ggn)\}\rightarrow 0$ for some index $\bar{l}$ implies $\mathrm{Re}\{\lambda_{l}(\Ggn)\}\rightarrow 0$ for $2\le l \le \bar{l}$, so we must check all eigenvalues $2\le l \le \bar{l}$. }

Now, $\left(\frac{\mathrm{Re}\{\lambda_l(\Ggn)\}}{\mathrm{Im}\{\lambda_l(\Ggn)\}}\right)^2 \le \mathrm{const.}$ is equivalent to having an upper bound on $\frac{\mathrm{Re}\{\lambda_l(\Ggn)\}}{\mathrm{Im}\{\lambda_l(\Ggn)\}}$ for an eigenvalue in the first quadrant (recall, the Laplacian eigenvalues appear in conjugate pairs in the RHP). This, in turn, is equivalent to having the argument $\mathrm{arg}\{\lambda_l(\Ggn)\}$ bounded away from zero. In other words, $\mathrm{arg}\{\lambda_l(\Ggn)> \psi$ for some fixed $\psi \in (0,\pi/2)$, and the theorem statement follows.
\end{proof}

\begin{figure}[t]
\centering
\includegraphics[scale=0.66]{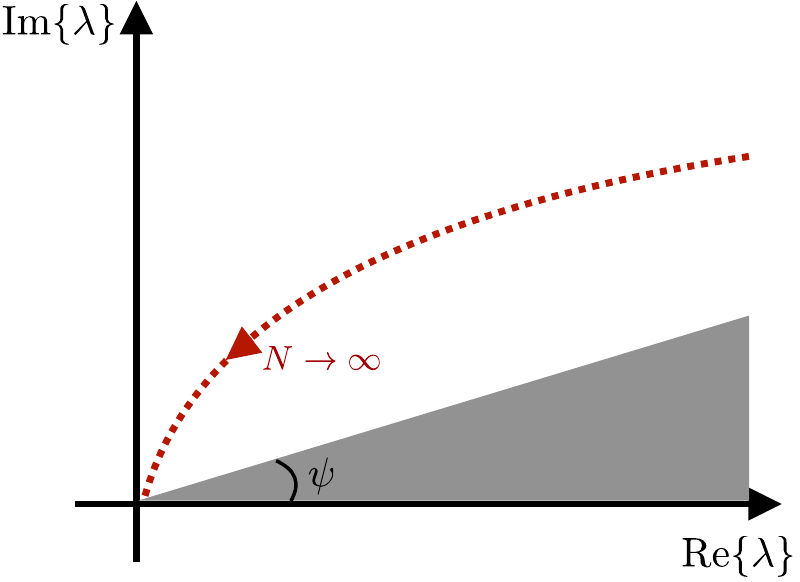}
\caption{{Illustration of the conditions in Theorem~\ref{thm:2ndorder}. If a Laplacian eigenvalue approaches the origin at an angle greater than some $\psi$ {as the network grows}, then second-order consensus {lacks scalable stability}. The example trajectory illustrates $\lambda_2$ of a family of directed ring graphs. }
 }
\label{fig:evcond}
\end{figure}

{A simpler statement pertaining to the special case of $\lambda_2$ can be stated as follows:}
\begin{corollary}
{If $n\ge 2$, no control on the form~\eqref{eq:consensuscompact}, subject to Assumption~\ref{ass:fixed}, 
is {scalably stable} in graph families where $\mathrm{Re}\{\lambda_2(\Ggn)\}\rightarrow 0$ as $N \rightarrow \infty$ while ${\mathrm{arg}\{\lambda_2(\Ggn)\}> {\psi} }$ for some constant ${\psi \in (0,\pi/2)}$ that is independent of~$N$.}
\end{corollary}

\subsection{Affected classes of graphs}
Theorem~\ref{thm:2ndorder} states that if at least one Laplacian eigenvalue is complex valued and approaches the origin at a non-zero angle as ${N\rightarrow \infty}$, then second-order consensus fails to be {scalably stable}. See also Figure~\ref{fig:evcond}. A particular graph family where this applies is directed ring graphs\footnote{More precisely, a ring graph that is not undirected.} with uniform edge weights, as already observed by~\citet{Cantos2016,HermanThesis,Stuedli2017}. Here, we demonstrate that it applies to the more general family of directed lattices with periodic boundary conditions.

\subsubsection{Directed periodic lattices}
Consider again {the $d$-dimensional $r$-fuzz lattice from Section~\ref{sec:graphs}}. We impose location-invariant edge weights in the sense that, if $d = 1$, $w_{i,i+k} = w_k$ for all $i \in {\mathbb{Z}_M}$, $k = \pm\{1,\ldots,r\}$. This means that the corresponding graph Laplacian for $d = 1$ is a $M\times M$ circulant matrix. In the higher-dimensional case, the Laplacian is the Kronecker sum of $d$ such matrices (since the graph is the Cartesian product of $d$ one-dimensional lattices) {and $N = M^d$}. Here, we assume the Laplacian is asymmetric:

\begin{ass}
The edge weight ${w_{k} \neq w_{-k}}$ for at least one $k \in \pm\{1,\ldots,r\}$.
\label{ass:directed} 
\end{ass}

\begin{lemma}
\label{lem:dirlattice}
For the {$d$-dimensional} $r$-fuzz lattice under Assumption~\ref{ass:directed}, \vspace{-1mm} 
\[\mathrm{Re}\{\lambda_2(\Gg_N)\} = \mathcal{O}\left( \frac{1}{N^{2/d}} \right),~~\mathrm{Im}\{\lambda_2(\Gg_N)\} =\mathcal{O}\left( \frac{1}{N^{1/d}} \right). \] 
\end{lemma}
\begin{proof}
The smallest (in real part) non-zero eigenvalue of the $r$-fuzz lattice is given by
\begin{equation}
\label{eq:l2lattice}
\lambda_2 = \sum_{\substack{k = -r  \\  k \neq 0} }^r \!\!w_i(1- \cos \left(\frac{2\pi k}{M} \right) )- \mathbf{j} \!\sum_{\substack{k = -r  \\  k \neq 0}}^r \!\! w_i\sin \left( \frac{2\pi k}{M} \right)\!,
\end{equation}
 where $M = N^{1/d}$ is the lattice size~\citep{Tegling2019}. The expression~\eqref{eq:l2lattice} is easily obtained from the case $d=1$, since the Laplacian eigenvalues of a Cartesian product of any two graphs are given by every possible sum of their respective Laplacian eigenvalues (see e.g.~\citet{Mohar1991}), and one eigenvalue is zero in each. Next, note that since $\sin(-x) = -\sin(x)$, it is only under Assumption~\ref{ass:directed} that $\mathrm{Im}\{\lambda_2\}\neq 0$. Finally, recalling that $r$ is bounded by Assumption~\ref{ass:q}, the lemma follows from Maclaurin series expansions of the real and imaginary parts. 
\end{proof}
Lemma~\ref{lem:dirlattice} implies that $\mathrm{arg}\{\lambda_2(\Gg_N)\} \rightarrow \pi/2$ as $N\rightarrow \infty$, so the conditions in Theorem~\ref{thm:2ndorder} clearly hold.  
\begin{rem}
In fact, \eqref{eq:l2lattice} will be an eigenvalue (though not necessarily~$\lambda_2$) of a graph that results from a Cartesian product of any graph with a $r$-fuzz {lattice}. This follows from the proof of Lemma~\ref{lem:dirlattice}. Such product graphs would thus also be affected by Theorem~\ref{thm:2ndorder}.
\end{rem}

\subsubsection{General necessary condition -- cyclicity}
Characterizing the Laplacian spectra of general directed graph families is a difficult and largely unsolved problem. Even determining the properties of graphs that have a real-valued spectrum, and which are therefore certainly \emph{not} affected by Theorem~\ref{thm:2ndorder}, is an open problem. 

A necessary condition, however, for $\Ggn$ having at least one complex eigenvalue is that $\Ggn$ has a directed cycle. This is, however, not sufficient. The term \emph{essentially cyclic graphs} has been proposed for graphs with non-real spectra, and properties of such graphs are examined in~\citet{Agaev2010}. To determine the eigenvalue behavior in~$N$ for families of such graphs, and thereby whether they are affected by Theorem~\ref{thm:2ndorder}, is an graph-theoretical endeavor that is outside the scope of the present paper.

\begin{figure}[t]
\centering
\subfloat[][{{Scalably stable}}]{
  \includegraphics[scale = 0.18]{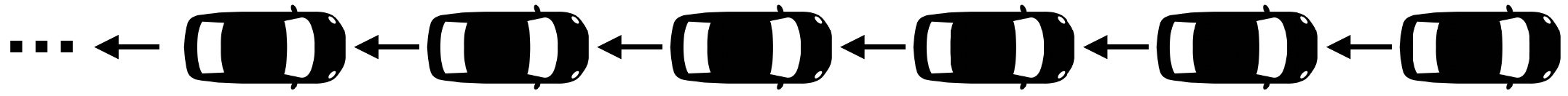}}
\hspace{6mm} 
  \subfloat[][{Not {scalably stable}}]{
  \includegraphics[scale = 0.18]{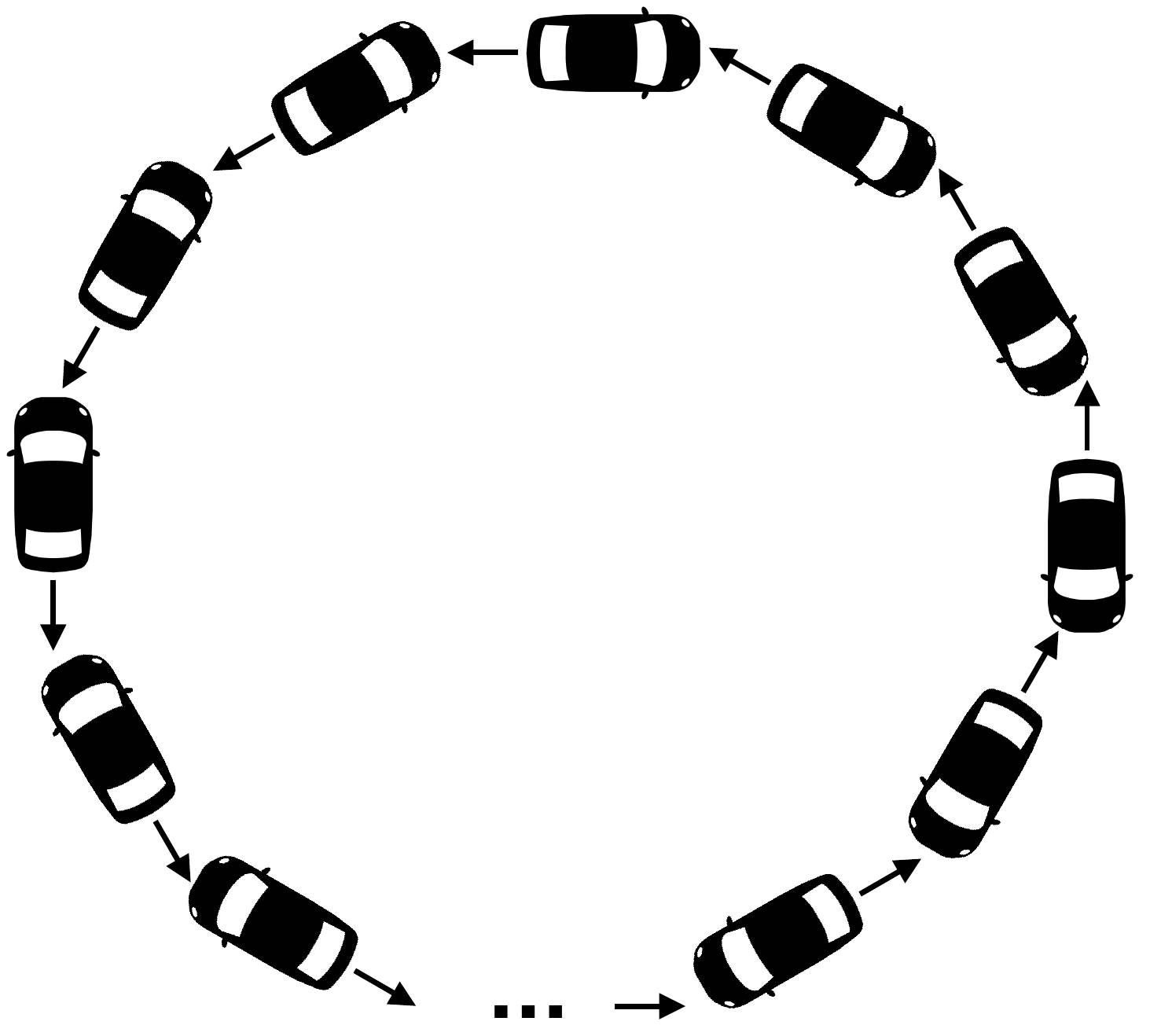}}
         \caption{
{Theorem~\ref{thm:2ndorder} reveals a scale fragility in the vehicle formation dynamics $\ddot{x}_i = -a_0(x_{i}-x_{i-1})-a_1(\dot{x}_{i}-\dot{x}_{i-1})$, where $x_i$ is vehicle $i$'s displacement. These dynamics can model adaptive cruise control in commercial vehicles~\citep{Gunter2021}. If the vehicles drive in a circle (let $x_{-1}=x_N$), the formation is destabilized at some size $\bar{N}$. The same issue does not apply to the line formation.  } }
\label{fig:linevsring}
\end{figure}

\subsection{Implications and numerical example}
\label{sec:ex2nd}
{These results} have interesting implications. First, that circular formations based on the consensus algorithm~\eqref{eq:consensuscompact} are scale fragile. 
For example, vehicles driving with adaptive cruise controllers available in modern commercial vehicles can indeed be modeled as our second-order consensus with unidirectional nearest-neighbor connections, see~\citet[\S II-A]{Gunter2021} (a constant reference spacing term can be eliminated by translating the state). If they drive in a circle, as on a ring road or as in many experimental set-ups (see e.g.~\citet{Stern2018}), our results show that the formation may be destabilized if too many vehicles join. See also Figure~\ref{fig:linevsring}. In such settings, however, {it can be possible to recover {scalable stability} by including absolute feedback.} 

Second, we can note that even if the feedback in a ring formation is bidirectional, that is, if the graph is undirected, it can be destabilized if a slight change in the weights renders the graph directed. Therefore, formations on undirected ring graphs are also fragile. We note that the same issues do not apply to formations on a line. The two therefore have fundamentally different scalability and robustness properties.

Figure~\ref{fig:2ndsim} shows a simulation of a growing circular vehicle formation to illustrate this section's results. Here, each vehicle's displacement $x_i$ is controlled with respect to the preceding vehicle so that $\ddot{x}_i = -w_{i,i-1}(x_{i}-x_{i-1})-3w_{i,i-1}(\dot{x}_{i}-\dot{x}_{i-1})$ for $i = 1,\ldots, N$. Let $x_{-1} = x_N$. We relax the assumption of location-invariant edge weights~$w_{ij}$ used for Lemma~\ref{lem:dirlattice}. Instead, as vehicles are added, the edge weights take random values in the interval $(0,1)$. In this example, the formation is destabilized at $\bar{N} = 14$ and the vehicles collide. 

\begin{figure}[t]
\centering
\subfloat[][{$N = 13$}]{
\begin{tikzpicture}
		\node[inner sep=0pt] (img) at (0,0)
		{ \includegraphics[width = 0.45\textwidth]{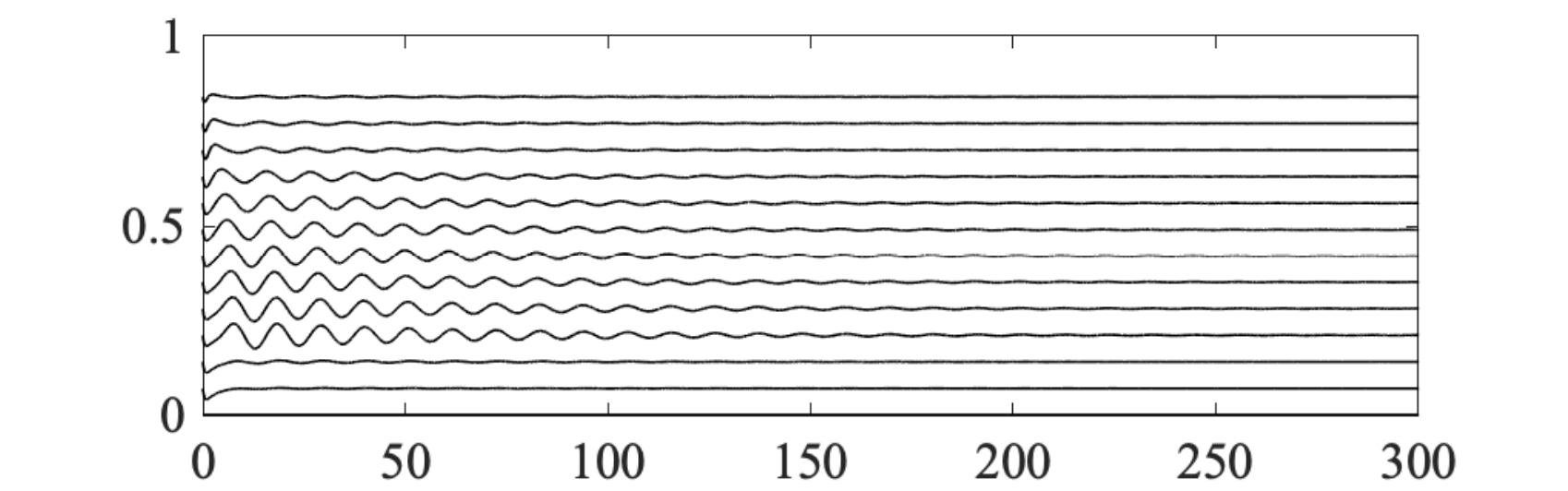}};
		\draw (0, -1.5) node {{\footnotesize Time $t$}};
		\node[xshift=-4cm,rotate=90,anchor=north]{{\footnotesize Rel. position}};
	\end{tikzpicture}
}
\\  \vspace{-2mm}
  \subfloat[][{$N = 14$}]{
  \begin{tikzpicture}
		\node[inner sep=0pt] (img) at (0,0)
		{\includegraphics[width = 0.45\textwidth]{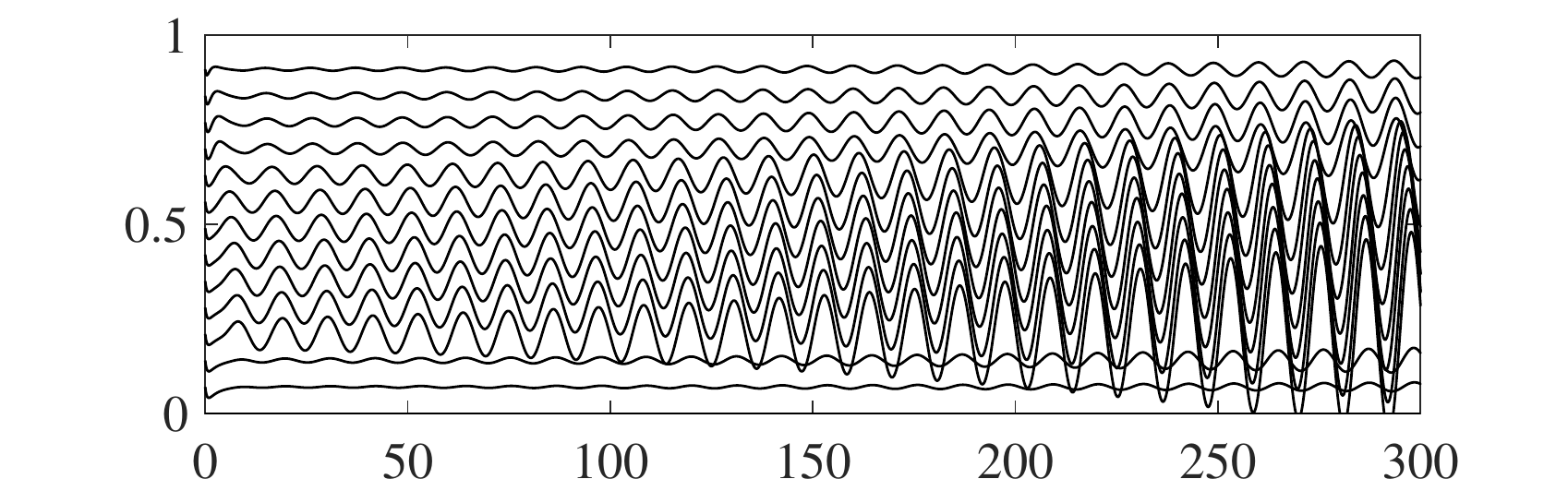}};
		\draw (0, -1.5) node {{\footnotesize Time $t$}};
		\node[xshift=-4cm,rotate=90,anchor=north]{{\footnotesize Rel. position}};
	\end{tikzpicture}
  }
         \caption{
{Simulation of a circular vehicle formation with unidirectional feedback and random edge weights. Each line represents a vehicle's position relative to vehicle~1, which has a step change in its velocity at time $t = 0$. The formation is stable with $N=13$, but the addition of a $14^{\mathrm{th}}$ vehicle destabilizes it. 
} }
\label{fig:2ndsim}
\end{figure}
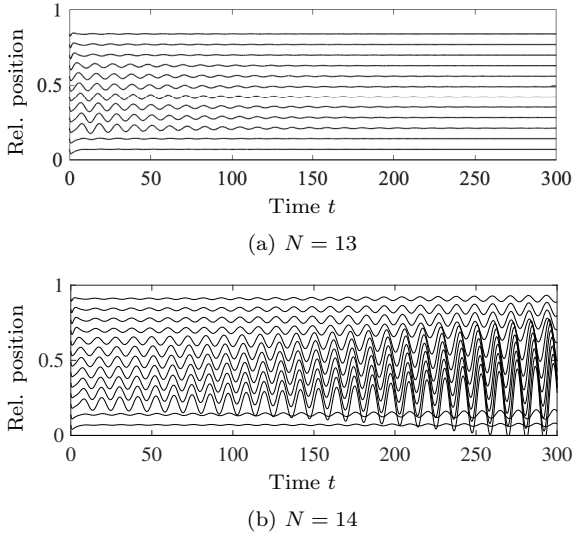

\section{Retrieving scalable stability}
\label{sec:mitigation}
Having been presented with fundamental limitations to the scalability of {modular, }localized consensus, an obvious question is how to change the algorithm, or relax assumptions on the network topology, to retrieve scalability. We next address this question {by pinning down on two key model assumptions.} 

\subsection{Relaxing the locality assumption}
\label{sec:locality}
Underlying our analysis was the assumption of locality in the sense of bounded nodal degrees, Assumption~\ref{ass:q}. Recall that under this assumption, all undirected graph families except expander families have decreasing algebraic connectivity and are thus affected by Theorem~\ref{thm:main}. If this assumption is relaxed, so that nodal neighborhoods are allowed to grow with $N$, the algebraic connectivity can remain bounded away from zero. {Scalable stability} can then be retrieved. Interestingly, {even though weights are fixed, }it can suffice to grow neighborhoods sub-linearly. 

We show this for a ring graph topology, but note that the same result applies to any graph that is better connected due to Lemma~\ref{lem:addedge}. 

\begin{theorem}
\label{thm:qscale}
Let $\{\Gg_N\}$ be a family of undirected 1-dimensional $q/2$-fuzz lattices ($q$ even), that is, ring graphs with edges between each node and its $q$ nearest neighbors. Then, if \[q \ge c N^{2/3} ,\]
with $c>0$ a constant independent of $N$, the sequence $\{\lambda_2(\Gg_N)\}$ is bounded away from zero as $N\rightarrow \infty$. 
\end{theorem}
\begin{proof}
The algebraic connectivity of $\Gg_N$ is \( \lambda_2(\Gg_N)= \sum_{k = -q/2}^{q/2} w(1-\cos\frac{2\pi k}{N}) \) if edge weights are uniform, i.e., $w_{ij} = w$ for all $(i,j)\in \Ee_N$.  
The derivation of this expression is based on the Discrete Fourier Transform, see e.g.~\citet{Tegling2019}. 
Therefore, in a graph with non-uniform weights, but with $w_{ij}\ge w_{\min}$, we have
\begin{align}
\nonumber
\lambda_2(\Gg_N)  & \ge  \sum_{k = -q/2}^{q/2}\!\! w_{\min}\left(1-\cos\frac{2\pi k}{N}\right)\\ \nonumber
 &= 2w_{\min}\!\left(1-\cos\frac{2\pi}{N}\right)+\cdots +2w_{\min}\!\left(1-\cos\frac{2\pi}{N}\right)\\ \nonumber
&\ge 2w_{\min} \frac{2}{\pi^2}\!\left( \!\!\left(\! \frac{2\pi}{N}\right)^2\! \!+\! \left( \frac{2\pi\cdot 2}{N}\right)^2\!\! + \!\cdots \!+ \!  \left(\! \frac{2\pi\cdot q}{N}\!\right)^2 \right)\\ \nonumber
&= \frac{16 w_{\min}}{N^2}\left(1^2 +2^2 +\cdots + q^2 \right)\\
&= \frac{16 w_{\min}}{N^2}\frac{q(q+1)(2q+1)}{6},
\label{eq:qboundderiv}
\end{align}
where the first inequality follows from Lemma~\ref{lem:addedge} and the second from the fact that $1-\cos x \ge \frac{2}{\pi^2x^2}$ for $x\in[-\pi,\pi]$. The last equality is a standard result for sums of sequences of squares.  Now, if $q\ge c N^{2/3}$, where $c$ is a positive constant, then \eqref{eq:qboundderiv} is lower bounded by \( \frac{16 w_{\min}}{N^2}\frac{2cN^2}{6} = \frac{16 c^3 w_{\min}}{6}, \) which is a positive constant independent of $N$. The theorem follows. 
\end{proof}

The sub-linear scaling in Theorem~\ref{thm:qscale} is surprising in light of well-known bounds on algebraic connectivity, which appear to require a linear scaling. One example is the {bound based on the }edge connectivity {$e(\Gg_N)$}: $\lambda_2(\Gg_N) \ge 2e(\Gg_N)(1-\cos\frac{\pi}{N})$ \citep[\S 4.3]{Fiedler1973}. Since the edge connectivity grows quadratically with the number of nearest-neighbor connections~$q$ and $(1-\cos\frac{\pi}{N})=\mathcal{O}( \frac{1}{N^2})$, this bound requires $q = \mathcal{O}(N)$. 

It is also notable that leader-follower consensus indeed requires a linear scaling of $q$. This is evident from~\eqref{eq:reducedcond}, which is a necessary stability condition. This again highlights an important difference in scalability between leaderless and leader-follower consensus. 

\begin{rem}
{Theorem~\ref{thm:qscale} is stated for a ring graph family that lets $\lambda_2(\Ggn)$ be expressed as a fairly simple function of~$q$ and $N$. Numerical evaluations show, however, that the same result holds in path graphs, see also Figure~\ref{fig:criticalN}. For more connected graph families, Lemma~\ref{lem:addedge} applies, making the result conservative. }
\end{rem}


\subsection{Impact of absolute feedback}
\label{sec:absolute}
{Scalable stability} can be retrieved if the control includes \emph{absolute} state feedback (equivalent to non-zero self-weights), {if this feedback is carefully designed}. {To highlight this result while keeping the section brief, we consider the case of $n=3$ and undirected graph families. } 
In this case, the control algorithm becomes 
\begin{equation}
\label{eq:consensusabsolute}
u_i = - \sum_{k = 0}^{2} \left( a_k\sum_{j \in \mathcal{N}_i} w_{ij}(x_i^{(k)} - x_j^{(k)}) - a_k^{\mathrm{abs}}x^{(k)} \right),
\end{equation}
and we say that absolute feedback from the state $x^{(k)}$ is available if one can set $a_k^{\mathrm{abs}}>0$. The closed-loop system dynamics become
\[ \ddt \xi =   \begin{bmatrix}
0 & I_N & 0 \\
0 & 0 & I_N \\
-a_0L \!-\!a_0^{\mathrm{abs}}I_N & -a_1L \!-\!a_1^{\mathrm{abs}}I_N& -a_2L \!-\!a_2^{\mathrm{abs}}I_N
\end{bmatrix}\!\! \xi.
\] 
The following proposition lines out {that absolute feedback from certain states is particularly important to retrieve scalable stability}. 
\begin{proposition}
Let $\{\Ggn\}$ be an undirected graph family in which $\{\lambda_2(\Ggn)\} \rightarrow 0$ as $N\rightarrow \infty$. Then, a necessary condition for {scalable stability} of the controller~\eqref{eq:consensusabsolute}, subject to Assumption A1, is that at least one of $a_1^{\mathrm{abs}}, a_2^{\mathrm{abs}} >0$.
\end{proposition}
\begin{proof}
The proof follows that of Theorem~\ref{thm:main}, with modifications lined out as follows. With absolute feedback terms, the characteristic polynomial corresponding to~\eqref{eq:charpoly} becomes
\begin{equation*}
p_l(s) = s^3 + (a_{2}\lambda_l + a_2^{\mathrm{abs}})s^2 + (a_{1}\lambda_l + a_1^{\mathrm{abs}})s + (a_{0}\lambda_l + a_0^{\mathrm{abs}}),
\end{equation*}
and the relevant stability condition is obtained from~\eqref{eq:RHcriterion} by substituting $(a_k\rell + a_k^{\mathrm{abs}})$ for $a_k\rell$. Since we let $\Ggn$ be undirected, $\lambda_l$ are real-valued, and the condition for $l = 2$ simplifies to
\begin{equation}
\label{eq:abscondreal}
(a_1\lambda_2 + a_1^{\mathrm{abs}})(a_2\lambda_2 + a_2^{\mathrm{abs}})  -a_0 \lambda_2 - a_0^{\mathrm{abs}} >0
\end{equation}
(which compares to~\eqref{eq:realRH}). If both $a_1^{\mathrm{abs}}=a_2^{\mathrm{abs}} =0$, \eqref{eq:abscondreal} is eventually violated as $\lambda_2\rightarrow 0$, regardless of $a_0^{\mathrm{abs}}$. However if at least one of $a_1^{\mathrm{abs}},a_2^{\mathrm{abs}} >0$ the condition can stay satisfied, e.g. if $a_1^{\mathrm{abs}}a_2>a_0$ or $a_2^{\mathrm{abs}}a_1>a_0$ while $a_0^{\mathrm{abs}}=0$. If both $a_1^{\mathrm{abs}},a_2^{\mathrm{abs}} >0$, it is also allowed to set $a_0^{\mathrm{abs}}>0$.
\end{proof}

This implies that absolute feedback from the high-order terms, that is, velocity or acceleration, is necessary to render the third-order consensus algorithm scalable. Reading the proof in more detail also reveals the interesting observation that absolute feedback from positions cannot be included unless there is also absolute feedback from \emph{both} velocity and acceleration (it will ruin {scalable stability} if included with only one of the two). This is somewhat counter-intuitive, as absolute feedback is usually beneficial for performance and stability, {though often more difficult to implement~\citep{Jensen2022}. }

\section{Discussion}
\label{sec:discussion}
This paper's results show that there is an important difference between the well-studied standard first-order consensus algorithm and the corresponding second- and higher-order algorithms, in that the latter are not always scalable {in a modular manner} to large networks. When subject to locality constraints, formally expressed through the network's Cheeger constant~\eqref{eq:cheeger}, high-order consensus will stop converging and become unstable at some finite network size. We remark that this result contradicts a statement made in~\citep[\S V]{Ren2007}, that convergence to consensus of a high-order multi-vehicle network ``will not be impacted as the number of vehicles increases'' {(though the authors clearly note that controller gains must be chosen to ensure stability.)}

Second-order consensus is subject to the same scale fragility in certain families of directed networks, such as directed ring graphs. 
An interesting consequence of both results is that, at some given network size, the addition of only one agent to a multi-agent network renders a previously converging system unstable. 
 This can be thought of as a type of phase transition. 
 {For open multi-agent systems~\citep{Hendrickx2017,Franceschelli2021} that obey a high-order consensus protocol, e.g. for flight formation, our results imply that special care must be taken to avoid this phase transition by limiting the network size or avoiding a localized network topology.   }
 We next discuss some further implications of our results. 

\subsection{Implications for distributed integral control}
If distributed integral control is applied to a lower-order consensus network {with relative feedback}, the closed-loop dynamics can be formulated analogously to the high-order consensus algorithm. Our results can be used to reveal conditions on such integral control for {scalable stability}. 

One example of such an integral controller is the distributed-averaging proportional-integral (DAPI) controller proposed for frequency control in electric power grids, see~\citet{Andreasson2014ACC,SimpsonPorco2013}. While in frequency control, absolute frequency feedback helps ensure {scalable stability}, the analogous control design based on relative feedback would {lack scalable stability}. 
In earlier work~\citep[Theorem 5.4]{Tegling2019} we have stated a particular stability result for distributed integral control, but the topic is far from fully explored.

\subsection{Asymptotic performance analysis}
A further interesting consequence of our results is that an analysis of the asymptotic (in network size) performance of localized, consensus-like feedback control is only possible in first- and second-order integrator networks. This means that the analysis on coherence scaling in large-scale networks in~\citet{Bamieh2012} cannot, as was conjectured there, be extended to chains of $n>2$ integrators. We also note that the analysis for second-order networks in that work hinges on the assumption of symmetric feedback, since the scale fragility from Theorem~\ref{thm:2ndorder} applies in directed tori.

\subsection{{Modular design vs. controller re-tuning}}
In order to be able to discuss a given controller's scalability in a network of increasing size, the assumption that it be fixed is necessary. {This presumes a modular design, implying }that the controller cannot be re-tuned as the network grows. By re-tuning the consensus algorithm from this paper, either by changing the gains~$a_k$, weights~$w_{ij}$, or by relaxing the locality assumption, consensus can be achieved also as the network grows. 

{Changing gains or weights would require adapting to the graph's changing algebraic connectivity. While this can indeed be estimated in a decentralized manner~\citep{Yang2010}, dynamic weight re-tuning algorithms as in~\citet{Kempton2018} require the entire network to participate in tuning to improve the connectivity. }
 Still, the design of controller re-tuning protocols~{-- which this paper shows to be necessary for scalable stability -- is a highly} interesting direction for future research.

\section*{Acknowledgements}
We would like to thank Swaroop Darbha, Federica Garin, Rick Middleton, Ali Jadbabaie, {and anonymous reviewers} for their insightful comments relating to various aspects of this work. 

This work was partially supported by the Wallenberg AI, Autonomous Systems and Software Program (WASP) funded by the Knut and Alice Wallenberg Foundation, the Swedish Research Council through grants 2016-00861 and 2019-00691, the Swedish Civil Contingencies Agency (MSB) through the project CERCES2, {as well as by the NSF through awards ECCS-1932777 and CMI- 1763064.}


\appendix
\section{Routh-Hurwitz criteria}
\label{RHapp}
{We state the Routh-Hurwitz criteria for polynomials with complex coefficients as they appear in~\citet[pp 21f]{Tondl1965}. }

\begin{applemma}
Consider the polynomial 
\begin{equation}
\label{eq:examplepoly}
p(\mu) = \mu^n + (f_{n-1} + \mathbf{j} g_{n-1})\mu^{n-1} +\ldots (f_0 + \mathbf{j} g_0) = 0,
\end{equation} 
where $\mathbf{j} = \sqrt{-1}$ denotes the imaginary unit.
The roots $\mu$ will be such that $\mathrm{Im}\{\mu\}>0$ if and only if all inequalities 
\begin{multline}
\label{eq:Tondlcrit}
-\Delta_2 = -\begin{vmatrix}
1 & f_{n\!-\!1}\\ 0 & g_{n\!-\!1}
\end{vmatrix} >0,~~\Delta_4 = \begin{vmatrix}
1 & f_{n\!-\!1} & f_{n\!-\!2} & f_{n\!-\!3} \\ 0 & g_{n\!-\!1}& g_{n\!-\!2} & g_{n\!-\!3} \\ 0 & 1 & f_{n\!-\!1} & f_{n\!-\!2}\\ 0& 0 & g_{n\!-\!1}& g_{n\!-\!2} \\
\end{vmatrix}\! >\!0,\\
,~\cdots~,~~~~~~~~~~~~~~~~~~~~~~~~~~~~~~~~~~~~~~~~~~~~~~\\ ~(-1)^n\Delta_{2n} \!= (-1)^n \!\begin{vmatrix}
1 & f_{n\!-\!1}& \cdots & f_0 & 0 &\cdots &\cdots & 0\\
0 & g_{n\!-\!1} & \cdots & g_0 & 0& \cdots &\cdots & 0\\
0 & 1 & \cdots & f_{1} & f_0 & 0 & \cdots & 0\\
0 & 0 & \cdots & g_{1} & g_0 & 0 & \cdots & 0\\
&&& \vdots &&& &\\
0 & \cdots & \cdots & 0 & 1 & \cdots & f_{1} & f_0\\
0 & \cdots & \cdots & 0 & 0 & \cdots & g_{1} & g_0\\
\end{vmatrix} \!>\!0
\end{multline}
are satisfied. 
\end{applemma} 
 
{Evaluating the determinants, the first two inequalities (which suffice to prove the theorems in this paper) read }   
\begin{align} \label{eq:firstcond}
g_{n\!-\!1}&<0,\\
\label{eq:secondcond}
 f_{n\!-\!1}g_{n\!-\!1}g_{n\!-\!2} - f_{n\!-\!2}^{\phantom{2}}g_{n\!-\!1}^2 + g_{n\!-\!3}g_{n\!-\!1} - g_{n\!-\!2}^2 &>0,
\end{align} 
{for $n\ge 3$. }
 

{Note that Lemma~A.1 gives a condition for all roots being in the upper half of the complex plane. To obtain a condition for poles in the left half plane ($\mathrm{Re}\{s\}<0$), we substitute $\mu = -\mathbf{j}s$ in~\eqref{eq:examplepoly} and identify the coefficients with the polynomial }
\begin{equation}
\label{eq:examplepoly2}
p(s) = s^n + b_{n-1}s^{n-1}+\ldots + b_1s + b_0.
\end{equation}
{ Those coefficients that appear in~\eqref{eq:firstcond}--\eqref{eq:secondcond} are then
$f_{n-1} = \mathrm{Im}\{ b_{n-1} \},$ $g_{n-1} = -\mathrm{Re}\{b_{n-1}\},$ $f_{n-2} = -\mathrm{Re}\{b_{n-2} \},$ $g_{n-2} = -\mathrm{Im}\{b_{n-2}\},$ $f_{n-3} = -\mathrm{Im}\{b_{n-3}\},$  $g_{n-3} = \mathrm{Re}\{b_{n-3}\}$. Note that these identifications hold regardless of~$n$, as the coefficient of the highest order term is set to~1 in both~\eqref{eq:examplepoly2} and \eqref{eq:examplepoly}.}

  \bibliographystyle{elsarticle-harv} 
\bibliography{emmasbib2015,BassamBib,consrefs}

\begin{thebibliography}{47}
\expandafter\ifx\csname natexlab\endcsname\relax\def\natexlab#1{#1}\fi
\providecommand{\url}[1]{\texttt{#1}}
\providecommand{\href}[2]{#2}
\providecommand{\path}[1]{#1}
\providecommand{\DOIprefix}{doi:}
\providecommand{\ArXivprefix}{arXiv:}
\providecommand{\URLprefix}{URL: }
\providecommand{\Pubmedprefix}{pmid:}
\providecommand{\doi}[1]{\href{http://dx.doi.org/#1}{\path{#1}}}
\providecommand{\Pubmed}[1]{\href{pmid:#1}{\path{#1}}}
\providecommand{\bibinfo}[2]{#2}
\ifx\xfnm\relax \def\xfnm[#1]{\unskip,\space#1}\fi
\bibitem[{Agaev and Chebotarev(2010)}]{Agaev2010}
\bibinfo{author}{Agaev, R.}, \bibinfo{author}{Chebotarev, P.},
  \bibinfo{year}{2010}.
\newblock \bibinfo{title}{Which digraphs with ring structure are essentially
  cyclic?}
\newblock \bibinfo{journal}{Advances in Applied Mathematics}
  \bibinfo{volume}{45}, \bibinfo{pages}{232 -- 251}.
\bibitem[{Andreasson et~al.(2014)Andreasson, Dimarogonas, Sandberg and
  Johansson}]{Andreasson2014ACC}
\bibinfo{author}{Andreasson, M.}, \bibinfo{author}{Dimarogonas, D.},
  \bibinfo{author}{Sandberg, H.}, \bibinfo{author}{Johansson, K.},
  \bibinfo{year}{2014}.
\newblock \bibinfo{title}{Distributed {PI}-control with applications to power
  systems frequency control}, in: \bibinfo{booktitle}{American Control Conf.},
  pp. \bibinfo{pages}{3183--3188}.
\bibitem[{Bamieh et~al.(2012)Bamieh, Jovanovi\'c, Mitra and
  Patterson}]{Bamieh2012}
\bibinfo{author}{Bamieh, B.}, \bibinfo{author}{Jovanovi\'c, M.R.},
  \bibinfo{author}{Mitra, P.}, \bibinfo{author}{Patterson, S.},
  \bibinfo{year}{2012}.
\newblock \bibinfo{title}{Coherence in large-scale networks:
  {D}imension-dependent limitations of local feedback}.
\newblock \bibinfo{journal}{IEEE Trans. Autom. Control} \bibinfo{volume}{57},
  \bibinfo{pages}{2235 --2249}.
\bibitem[{Barooah and Hespanha(2005)}]{Barooah2005}
\bibinfo{author}{Barooah, P.}, \bibinfo{author}{Hespanha, J.P.},
  \bibinfo{year}{2005}.
\newblock \bibinfo{title}{Error amplification and disturbance propagation in
  vehicle strings with decentralized linear control}, in:
  \bibinfo{booktitle}{IEEE Conf. on Decision and Control (CDC)}, pp.
  \bibinfo{pages}{4964--4969}.
\bibitem[{Besselink and Knorn(2018)}]{Besselink2018}
\bibinfo{author}{Besselink, B.}, \bibinfo{author}{Knorn, S.},
  \bibinfo{year}{2018}.
\newblock \bibinfo{title}{Scalable input-to-state stability for performance
  analysis of large-scale networks}.
\newblock \bibinfo{journal}{IEEE Control Syst. Lett.} \bibinfo{volume}{2},
  \bibinfo{pages}{507--512}.
\bibitem[{Brouwer and Haemers(2012)}]{BrouwerBook}
\bibinfo{author}{Brouwer, A.E.}, \bibinfo{author}{Haemers, W.H.},
  \bibinfo{year}{2012}.
\newblock \bibinfo{title}{Spectra of Graphs}.
\newblock \bibinfo{address}{New York, NY}.
\bibitem[{Cantos et~al.(2016)Cantos, Veerman and Hammond}]{Cantos2016}
\bibinfo{author}{Cantos, C.}, \bibinfo{author}{Veerman, J.},
  \bibinfo{author}{Hammond, D.}, \bibinfo{year}{2016}.
\newblock \bibinfo{title}{Signal velocity in oscillator arrays}.
\newblock \bibinfo{journal}{Eur. Phys. J. Spec.} \bibinfo{volume}{225},
  \bibinfo{pages}{1115--1126}.
\bibitem[{Chung(1997)}]{ChungBook}
\bibinfo{author}{Chung, F.}, \bibinfo{year}{1997}.
\newblock \bibinfo{title}{Spectral Graph Theory}.
\newblock \bibinfo{address}{Providence, RI}.
\bibitem[{Chung(2005)}]{Chung2005}
\bibinfo{author}{Chung, F.}, \bibinfo{year}{2005}.
\newblock \bibinfo{title}{Laplacians and the {C}heeger inequality for directed
  graphs}.
\newblock \bibinfo{journal}{Annals of Combinatorics} \bibinfo{volume}{9},
  \bibinfo{pages}{1--19}.
\bibitem[{Fax and Murray(2004)}]{FaxMurray}
\bibinfo{author}{Fax, J.A.}, \bibinfo{author}{Murray, R.M.},
  \bibinfo{year}{2004}.
\newblock \bibinfo{title}{Information flow and cooperative control of vehicle
  formations}.
\newblock \bibinfo{journal}{IEEE Trans. Autom. Control} \bibinfo{volume}{49},
  \bibinfo{pages}{1465--1476}.
\bibitem[{Fiedler(1973)}]{Fiedler1973}
\bibinfo{author}{Fiedler, M.}, \bibinfo{year}{1973}.
\newblock \bibinfo{title}{Algebraic connectivity of graphs}.
\newblock \bibinfo{journal}{Czechoslovak Mathematical Journal}
  \bibinfo{volume}{23}, \bibinfo{pages}{298–305}.
\bibitem[{Franceschelli and Frasca(2021)}]{Franceschelli2021}
\bibinfo{author}{Franceschelli, M.}, \bibinfo{author}{Frasca, P.},
  \bibinfo{year}{2021}.
\newblock \bibinfo{title}{Stability of open multiagent systems and applications
  to dynamic consensus}.
\newblock \bibinfo{journal}{IEEE Trans. Autom. Control} \bibinfo{volume}{66},
  \bibinfo{pages}{2326--2331}.
\bibitem[{Friedman(1991)}]{Friedman1991}
\bibinfo{author}{Friedman, J.}, \bibinfo{year}{1991}.
\newblock \bibinfo{title}{On the second eigenvalue and random walks in random
  $d$-regular graphs}.
\newblock \bibinfo{journal}{Combinatorica} \bibinfo{volume}{11},
  \bibinfo{pages}{331--362}.
\bibitem[{Grone et~al.(1990)Grone, Merris and Sunder}]{Grone1990}
\bibinfo{author}{Grone, R.}, \bibinfo{author}{Merris, R.},
  \bibinfo{author}{Sunder, V.}, \bibinfo{year}{1990}.
\newblock \bibinfo{title}{The {L}aplacian spectrum of a graph}.
\newblock \bibinfo{journal}{SIAM J. Matrix Anal. Appl.} \bibinfo{volume}{11},
  \bibinfo{pages}{218--238}.
\bibitem[{Gunter et~al.(2021)Gunter, Gloudemans, Stern, McQuade, Bhadani,
  Bunting and {\textit{et al.}}}]{Gunter2021}
\bibinfo{author}{Gunter, G.}, \bibinfo{author}{Gloudemans, D.},
  \bibinfo{author}{Stern, R.E.}, \bibinfo{author}{McQuade, S.},
  \bibinfo{author}{Bhadani, R.}, \bibinfo{author}{Bunting, M.},
  \bibinfo{author}{{\textit{et al.}}}, \bibinfo{year}{2021}.
\newblock \bibinfo{title}{Are commercially implemented adaptive cruise control
  systems string stable?}
\newblock \bibinfo{journal}{IEEE Trans. Intell. Transp. Syst.}
  \bibinfo{volume}{22}, \bibinfo{pages}{6992--7003}.
\bibitem[{Hendrickx and Martin(2017)}]{Hendrickx2017}
\bibinfo{author}{Hendrickx, J.M.}, \bibinfo{author}{Martin, S.},
  \bibinfo{year}{2017}.
\newblock \bibinfo{title}{Open multi-agent systems: Gossiping with random
  arrivals and departures}, in: \bibinfo{booktitle}{IEEE Conf. on Decision and
  Control (CDC)}, pp. \bibinfo{pages}{763--768}.
\bibitem[{Herman(2016)}]{HermanThesis}
\bibinfo{author}{Herman, I.}, \bibinfo{year}{2016}.
\newblock \bibinfo{title}{Scaling in vehicle platoons}.
\newblock \bibinfo{type}{Phd thesis}. {C}zech {T}echnical {U}niversity in
  {P}rague.
\newblock \URLprefix
  \url{https://support.dce.felk.cvut.cz/mediawiki/images/d/d1/Diz\_2017\_herman\_ivo.pdf}.
\bibitem[{Horn and Johnson(1985)}]{HornJohnson}
\bibinfo{author}{Horn, R.A.}, \bibinfo{author}{Johnson, C.R.},
  \bibinfo{year}{1985}.
\newblock \bibinfo{title}{Matrix Analysis}.
\newblock \bibinfo{publisher}{Cambridge University Presss},
  \bibinfo{address}{New York}.
\bibitem[{Jadbabaie et~al.(2003)Jadbabaie, Lin and Morse}]{Jadbabaie2003}
\bibinfo{author}{Jadbabaie, A.}, \bibinfo{author}{Lin, J.},
  \bibinfo{author}{Morse, A.S.}, \bibinfo{year}{2003}.
\newblock \bibinfo{title}{Coordination of groups of mobile autonomous agents
  using nearest neighbor rules}.
\newblock \bibinfo{journal}{IEEE Trans. Autom. Control} \bibinfo{volume}{48},
  \bibinfo{pages}{988--1001}.
\bibitem[{Jensen and Bamieh(2022)}]{Jensen2022}
\bibinfo{author}{Jensen, E.}, \bibinfo{author}{Bamieh, B.},
  \bibinfo{year}{2022}.
\newblock \bibinfo{title}{On structured-closed-loop versus
  structured-controller design: the case of relative measurement feedback}.
\newblock \href{http://arxiv.org/abs/2008.11291}{{\tt arXiv:2008.11291}}.
\bibitem[{Jiang et~al.(2009)Jiang, Wang and Jia}]{Jiang2009}
\bibinfo{author}{Jiang, F.}, \bibinfo{author}{Wang, L.}, \bibinfo{author}{Jia,
  Y.}, \bibinfo{year}{2009}.
\newblock \bibinfo{title}{Consensus in leaderless networks of
  high-order-integrator agents}, in: \bibinfo{booktitle}{American Control
  Conf.}, pp. \bibinfo{pages}{4458--4463}.
\bibitem[{Kempton et~al.(2018)Kempton, Herrmann and Bernardo}]{Kempton2018}
\bibinfo{author}{Kempton, L.}, \bibinfo{author}{Herrmann, G.},
  \bibinfo{author}{Bernardo, M.d.}, \bibinfo{year}{2018}.
\newblock \bibinfo{title}{Self-organization of weighted networks for optimal
  synchronizability}.
\newblock \bibinfo{journal}{IEEE Trans. Control Netw. Syst.}
  \bibinfo{volume}{5}, \bibinfo{pages}{1541--1550}.
\bibitem[{Krebs and Shaheen(2011)}]{KrebsBook}
\bibinfo{author}{Krebs, M.}, \bibinfo{author}{Shaheen, A.},
  \bibinfo{year}{2011}.
\newblock \bibinfo{title}{Expander Families and {C}ayley Graphs, A beginner's
  guide}.
\newblock \bibinfo{publisher}{Oxford University Press},
  \bibinfo{address}{Oxford}.
\bibitem[{Mohar(1991)}]{Mohar1991}
\bibinfo{author}{Mohar, B.}, \bibinfo{year}{1991}.
\newblock \bibinfo{title}{The {L}aplacian spectrum of graphs}, in:
  \bibinfo{booktitle}{Graph Theory, Combinatorics, and Applications},
  \bibinfo{publisher}{Wiley}. pp. \bibinfo{pages}{871--898}.
\bibitem[{Ni and Cheng(2010)}]{Ni2010}
\bibinfo{author}{Ni, W.}, \bibinfo{author}{Cheng, D.}, \bibinfo{year}{2010}.
\newblock \bibinfo{title}{Leader-following consensus of multi-agent systems
  under fixed and switching topologies}.
\newblock \bibinfo{journal}{Syst. Control Lett.} \bibinfo{volume}{59},
  \bibinfo{pages}{209 -- 217}.
\bibitem[{Olfati-Saber(2006)}]{OlfatiSaber2006}
\bibinfo{author}{Olfati-Saber, R.}, \bibinfo{year}{2006}.
\newblock \bibinfo{title}{Flocking for multi-agent dynamic systems: algorithms
  and theory}.
\newblock \bibinfo{journal}{IEEE Trans. Autom. Control} \bibinfo{volume}{51},
  \bibinfo{pages}{401--420}.
\bibitem[{Olfati-Saber et~al.(2007)Olfati-Saber, Fax and
  Murray}]{OlfatiSaber2007}
\bibinfo{author}{Olfati-Saber, R.}, \bibinfo{author}{Fax, J.A.},
  \bibinfo{author}{Murray, R.M.}, \bibinfo{year}{2007}.
\newblock \bibinfo{title}{Consensus and cooperation in networked multi-agent
  systems}.
\newblock \bibinfo{journal}{Proc. of the IEEE} \bibinfo{volume}{95},
  \bibinfo{pages}{215--233}.
\bibitem[{Olfati-Saber and Murray(2004)}]{OlfatiSaber2004}
\bibinfo{author}{Olfati-Saber, R.}, \bibinfo{author}{Murray, R.M.},
  \bibinfo{year}{2004}.
\newblock \bibinfo{title}{Consensus problems in networks of agents with
  switching topology and time-delays}.
\newblock \bibinfo{journal}{IEEE Trans. Autom. Control} \bibinfo{volume}{49},
  \bibinfo{pages}{1520--1533}.
\bibitem[{Patterson and Bamieh(2014)}]{Patterson2014}
\bibinfo{author}{Patterson, S.}, \bibinfo{author}{Bamieh, B.},
  \bibinfo{year}{2014}.
\newblock \bibinfo{title}{Consensus and coherence in fractal networks}.
\newblock \bibinfo{journal}{IEEE Trans. Control Netw. Syst.}
  \bibinfo{volume}{1}, \bibinfo{pages}{338--348}.
\bibitem[{Radmanesh et~al.(2017)Radmanesh, Naghash and
  Mohamadifard}]{Radmanesh2017}
\bibinfo{author}{Radmanesh, A.}, \bibinfo{author}{Naghash, A.},
  \bibinfo{author}{Mohamadifard, A.}, \bibinfo{year}{2017}.
\newblock \bibinfo{title}{Optimal distributed control of multi agents:
  Generalization of consensus algorithms for high-order state derivatives of
  {SISO} and {MIMO} systems}, in: \bibinfo{booktitle}{International Conf. on
  Control, Automation and Robotics}, pp. \bibinfo{pages}{606--611}.
\bibitem[{Ren et~al.(2007)Ren, Moore and Chen}]{Ren2007}
\bibinfo{author}{Ren, W.}, \bibinfo{author}{Moore, K.L.},
  \bibinfo{author}{Chen, Y.}, \bibinfo{year}{2007}.
\newblock \bibinfo{title}{High-order and model reference consensus algorithms
  in cooperative control of multi-vehicle systems}.
\newblock \bibinfo{journal}{J. Dyn. Syst. Meas. Control} \bibinfo{volume}{129},
  \bibinfo{pages}{678--688}.
\bibitem[{Rezaee and Abdollahi(2015)}]{Rezaee2015}
\bibinfo{author}{Rezaee, H.}, \bibinfo{author}{Abdollahi, F.},
  \bibinfo{year}{2015}.
\newblock \bibinfo{title}{Average consensus over high-order multiagent
  systems}.
\newblock \bibinfo{journal}{IEEE Trans. Autom. Control} \bibinfo{volume}{60},
  \bibinfo{pages}{3047--3052}.
\bibitem[{Seiler et~al.(2004)Seiler, Pant and Hedrick}]{Seiler2004}
\bibinfo{author}{Seiler, P.}, \bibinfo{author}{Pant, A.},
  \bibinfo{author}{Hedrick, K.}, \bibinfo{year}{2004}.
\newblock \bibinfo{title}{Disturbance propagation in vehicle strings}.
\newblock \bibinfo{journal}{IEEE Trans. Autom. Control} \bibinfo{volume}{49},
  \bibinfo{pages}{1835--1842}.
\bibitem[{Siami and Motee(2016)}]{SiamiMotee2015}
\bibinfo{author}{Siami, M.}, \bibinfo{author}{Motee, N.}, \bibinfo{year}{2016}.
\newblock \bibinfo{title}{Fundamental limits and tradeoffs on disturbance
  propagation in large-scale dynamical networks}.
\newblock \bibinfo{journal}{IEEE Trans. Autom. Control} \bibinfo{volume}{61},
  \bibinfo{pages}{4055--4062}.
\bibitem[{Simpson-Porco et~al.(2013)Simpson-Porco, D\"orfler and
  Bullo}]{SimpsonPorco2013}
\bibinfo{author}{Simpson-Porco, J.W.}, \bibinfo{author}{D\"orfler, F.},
  \bibinfo{author}{Bullo, F.}, \bibinfo{year}{2013}.
\newblock \bibinfo{title}{Synchronization and power sharing for
  droop-controlled inverters in islanded microgrids}.
\newblock \bibinfo{journal}{Automatica} \bibinfo{volume}{49},
  \bibinfo{pages}{2603 -- 2611}.
\bibitem[{Spielman and Teng(2007)}]{Spielman2007}
\bibinfo{author}{Spielman, D.A.}, \bibinfo{author}{Teng, S.H.},
  \bibinfo{year}{2007}.
\newblock \bibinfo{title}{Spectral partitioning works: Planar graphs and finite
  element meshes}.
\newblock \bibinfo{journal}{Linear Algebra Appl.} \bibinfo{volume}{421},
  \bibinfo{pages}{284 -- 305}.
\newblock \bibinfo{note}{Special Issue in honor of Miroslav Fiedler}.
\bibitem[{Stern et~al.(2018)}]{Stern2018}
\bibinfo{author}{Stern, R.E.}, et~al., \bibinfo{year}{2018}.
\newblock \bibinfo{title}{Dissipation of stop-and-go waves via control of
  autonomous vehicles: Field experiments}.
\newblock \bibinfo{journal}{Transportation Research Part C: Emerging
  Technologies} \bibinfo{volume}{89}, \bibinfo{pages}{205 -- 221}.
\bibitem[{St\"udli et~al.(2017)St\"udli, Seron and Middleton}]{Stuedli2017}
\bibinfo{author}{St\"udli, S.}, \bibinfo{author}{Seron, M.M.},
  \bibinfo{author}{Middleton, R.H.}, \bibinfo{year}{2017}.
\newblock \bibinfo{title}{Vehicular platoons in cyclic interconnections with
  constant inter-vehicle spacing}, in: \bibinfo{booktitle}{20th IFAC World
  Congress}, pp. \bibinfo{pages}{2511 -- 2516}.
\bibitem[{Swaroop and Hedrick(1996)}]{Swaroop1996}
\bibinfo{author}{Swaroop, D.}, \bibinfo{author}{Hedrick, J.K.},
  \bibinfo{year}{1996}.
\newblock \bibinfo{title}{String stability of interconnected systems}.
\newblock \bibinfo{journal}{IEEE Trans. Autom. Control} \bibinfo{volume}{41},
  \bibinfo{pages}{349--357}.
\bibitem[{Tegling et~al.(2019a)Tegling, Bamieh and Sandberg}]{Tegling2019ACC}
\bibinfo{author}{Tegling, E.}, \bibinfo{author}{Bamieh, B.},
  \bibinfo{author}{Sandberg, H.}, \bibinfo{year}{2019}a.
\newblock \bibinfo{title}{Localized high-order consensus destabilizes
  large-scale networks}, in: \bibinfo{booktitle}{American Control Conf. (ACC)},
  pp. \bibinfo{pages}{760--765}.
\bibitem[{Tegling et~al.(2019b)Tegling, Middleton and
  Seron}]{Tegling2019NECSYS}
\bibinfo{author}{Tegling, E.}, \bibinfo{author}{Middleton, R.H.},
  \bibinfo{author}{Seron, M.M.}, \bibinfo{year}{2019}b.
\newblock \bibinfo{title}{Scalability and fragility in bounded-degree consensus
  networks}, in: \bibinfo{booktitle}{8th IFAC Workshop on Distributed
  Estimation and Control in Networked Systems (NecSys)}.
\bibitem[{{Tegling} et~al.(2019){Tegling}, {Mitra}, {Sandberg} and
  {Bamieh}}]{Tegling2019}
\bibinfo{author}{{Tegling}, E.}, \bibinfo{author}{{Mitra}, P.},
  \bibinfo{author}{{Sandberg}, H.}, \bibinfo{author}{{Bamieh}, B.},
  \bibinfo{year}{2019}.
\newblock \bibinfo{title}{On fundamental limitations of dynamic feedback
  control in regular large-scale networks}.
\newblock \bibinfo{journal}{IEEE Trans. Autom. Control} \bibinfo{volume}{64},
  \bibinfo{pages}{4936--4951}.
\bibitem[{Tondl(1965)}]{Tondl1965}
\bibinfo{author}{Tondl, A.}, \bibinfo{year}{1965}.
\newblock \bibinfo{title}{Some problems of rotor dynamics}.
\newblock \bibinfo{publisher}{Czechoslovak {A}cademy of {S}ciences},
  \bibinfo{address}{Prague}.
\bibitem[{Xia and Cao(2017)}]{Xia2016}
\bibinfo{author}{Xia, W.}, \bibinfo{author}{Cao, M.}, \bibinfo{year}{2017}.
\newblock \bibinfo{title}{Analysis and applications of spectral properties of
  grounded {L}aplacian matrices for directed networks}.
\newblock \bibinfo{journal}{Automatica} \bibinfo{volume}{80},
  \bibinfo{pages}{10 -- 16}.
\bibitem[{Yadlapalli et~al.(2006)Yadlapalli, Darbha and
  Rajagopal}]{Swaroop2006}
\bibinfo{author}{Yadlapalli, S.K.}, \bibinfo{author}{Darbha, S.},
  \bibinfo{author}{Rajagopal, K.R.}, \bibinfo{year}{2006}.
\newblock \bibinfo{title}{Information flow and its relation to stability of the
  motion of vehicles in a rigid formation}.
\newblock \bibinfo{journal}{IEEE Trans. Autom. Control} \bibinfo{volume}{51},
  \bibinfo{pages}{1315--1319}.
\bibitem[{Yang et~al.(2010)Yang, Freeman, Gordon, Lynch, Srinivasa and
  Sukthankar}]{Yang2010}
\bibinfo{author}{Yang, P.}, \bibinfo{author}{Freeman, R.},
  \bibinfo{author}{Gordon, G.}, \bibinfo{author}{Lynch, K.},
  \bibinfo{author}{Srinivasa, S.}, \bibinfo{author}{Sukthankar, R.},
  \bibinfo{year}{2010}.
\newblock \bibinfo{title}{Decentralized estimation and control of graph
  connectivity for mobile sensor networks}.
\newblock \bibinfo{journal}{Automatica} \bibinfo{volume}{46},
  \bibinfo{pages}{390--396}.
\bibitem[{Zuo et~al.(2018)Zuo, Tian, Defoort and Ding}]{Zuo2017}
\bibinfo{author}{Zuo, Z.}, \bibinfo{author}{Tian, B.},
  \bibinfo{author}{Defoort, M.}, \bibinfo{author}{Ding, Z.},
  \bibinfo{year}{2018}.
\newblock \bibinfo{title}{Fixed-time consensus tracking for multi-agent systems
  with high-order integrator dynamics}.
\newblock \bibinfo{journal}{IEEE Trans. Autom. Control} \bibinfo{volume}{63},
  \bibinfo{pages}{563--570}.

\end{thebibliography}





\end{document}